\documentclass[a4paper,fleqn]{cas-sc}

\usepackage[authoryear,longnamesfirst]{natbib}

\usepackage{lipsum}
\usepackage{amsfonts}
\usepackage{graphicx}
\usepackage{epstopdf}
\usepackage{algorithm}
\usepackage{bm}

\usepackage{amsthm}
\newtheorem{definition}{Definition}
\newtheorem{mytheorem}{Theorem}
\newtheorem{lemma}{Lemma}
\newtheorem{corollary}{Corollary}

\usepackage{tikz}
\usetikzlibrary{matrix}
\usepackage{nicematrix, tikz}

\usepackage{todonotes}

\def\bbR{{\mathbb R}}
\def\BM{{\mathrm{BM}}_{(\ell,m)}^{n \times n}}
\def\SBM{{\mathrm{SBM}}_{\ell}^{n \times n}}
\def\BPS{{\mathrm{BPS}}_{(r,p),(\ell,m)}^{n \times n}}
\def\SBPS{{\mathrm{SBPS}}_{r,\ell}^{n \times n}}
\def\SBPS{{\mathrm{SBPS}}_{r,\ell}^{n \times n}}
\def\calP{{\mathcal P}}

\def\calH{{\mathcal H}}

\def\calS{{\mathcal S}}
\def\bfx{{\mathbf x}}
\def\bfy{{\mathbf y}}
\def\bfe{{\mathbf e}}
\def\bfb{{\mathbf b}}

\def\bfs{{\mathbf s}}
\def\bfv{{\mathbf v}}

\def\bfw{{\mathbf w}}

\def\bft{{\mathbf t}}
\def\bfr{{\mathbf r}}
\def\bfk{{\mathbf k}}
\def\bftau{{\bm\tau}}
\def\bfzero{{\bm0}}

\makeatletter
\@addtoreset{algorithm}{subsection}
\makeatother
\usepackage{algpseudocode}
\newtheorem{remark}{Remark}[section]

\usepackage{amsopn}
\DeclareMathOperator{\diag}{diag}
\DeclareMathOperator{\triu}{triu}
\DeclareMathOperator{\tril}{tril}

\ifpdf
  \DeclareGraphicsExtensions{.eps,.pdf,.png,.jpg}
\else
  \DeclareGraphicsExtensions{.eps}
\fi

\def\tsc#1{\csdef{#1}{\textsc{\lowercase{#1}}\xspace}}
\tsc{WGM}
\tsc{QE}
\tsc{EP}
\tsc{PMS}
\tsc{BEC}
\tsc{DE}

\begin{document}
\let\WriteBookmarks\relax
\def\floatpagepagefraction{1}
\def\textpagefraction{.001}

\shorttitle{QR for Banded-Plus-Semiseparable Matrices}

\shortauthors{T. Chen and S. Olver}

\title [mode = title]{The QR factorization of a banded-plus-semiseparable matrix is computable in linear complexity with Householder reflections}                      
\tnotemark[1,2]


\author[1]{Tao Chen}
\cormark[1] 
\ead{t.chen24@imperial.ac.uk}

\author[1]{ Sheehan Olver} 
\ead{s.olver@imperial.ac.uk}

\affiliation[1]{organization={Department of Mathematics, Imperial College London},
    addressline={South Kensington Campus}, 
    city={London},
    postcode={SW7 2AZ}, 
    country={United Kingdom}}

\cortext[cor1]{Corresponding author}

\begin{abstract}
We show that the QR factorization of a banded-plus-semiseparable (BPS) matrix is computable in optimal linear complexity with respect to the discretization size by showing that the intermediate stages of a QR factorization as computed using Householder reflection maintain a specific structure which has optimal storage.  While optimal complexity QR factorizations of BPS matrices are known via Given's rotations, our Householder-based framework enables adaptive, partial QR factorizations whose computational cost depends solely on the number of upper-triangularized columns rather than the global matrix size and represents the factors in a format consistent with LAPACK. This allows users to switch to highly optimized dense matrix operations for the final solve, creating a fast hybrid approach. We further deduce that the Compact WY format also has BPS structure and can be computed in optimal complexity.  Finally, for symmetric BPS matrices, we show that the reverse $RQ$ product preserves the BPS structure, establishing the algebraic closure of this matrix class under orthogonal similarity transformations and facilitating efficient chain-structured matrix decompositions. Numerical experiments validate the optimal linear complexity, confirm high numerical accuracy, demonstrate excellent computational efficiency via LAPACK integration, and show substantial speedups compared with existing hierarchical approaches. The algorithms have been implemented in an open-source Julia package, providing an efficient and accessible platform for practical use. 
\end{abstract}


\begin{highlights}
\item We propose a fast QR factorization algorithm for banded-plus-semiseparable matrices using Householder transformations.
\item The proposed algorithm achieves a strict linear computational complexity and effectively preserves the underlying matrix structures throughout the process.
\item At each stage the partial QR factorization has a precise structured representation to represent the Householder factors in a way that is compatible with LAPack's unblocked QR routine.
\item The Compact WY format also bases banded-plus-semiseparable structure that can be computed in optimal complexity.
\item A high-performance, open-source Julia package is developed and validated through extensive numerical benchmarks.

\end{highlights}


\maketitle

\section{Introduction}

Banded-plus-semiseparable (BPS) matrices, expressible as
\begin{equation*}
A = \underbrace{B}_{\text{banded}} + \underbrace{\tril(UV^\top, -1)}_{\text{lower semiseparable rank } r} + \underbrace{\triu(WS^\top, 1)}_{\text{upper semiseparable rank } p} \in \bbR^{n \times n},
\end{equation*}
where $U ,V \in \bbR^{n \times r}$ and $W,S \in \bbR^{n\times p}$
 arise in numerous applications; from spectral methods for differential equations~\citep{iserles2025stable} to signal processing, control theory, and eigenvalue problems~\citep{vandebril2005note}. Importantly, the inverse of a banded matrix is itself a BPS matrix, making this structure fundamental for representing inverses of banded matrices~\citep{rozsa1991band}. Their structure requires only $O(n)$ storage as $n \rightarrow \infty$, which invites the development of \( O(n) \) algorithms.

Pioneering work has established stable $O(n)$ direct solvers for various classes of BPS and related rank-structured matrices. In particular, a mature and extensive literature has leveraged sequences of Givens rotations to maintain these compact representations in QR-based linear solvers. The foundational theory and rotational ordering patterns for generic rank-structured systems are comprehensively detailed in the seminal books of~\cite{vandebril2008matrix} and~\cite{eidelman2014separable}, both of which provide exhaustive analyses of fast solvers, including QR and LU-type factorizations that adapt to BPS forms, see  \cite[p377]{vandebril2008matrix}. The \texttt{SSPACK} library by \cite{vandebrilmanual} delivers a robust Matlab framework for working with semiseparable and related matrices, though only supports QR-based solves for diagonal-plus-semiseparable matrices (see \cite{eidelman1997inversion} for details of this special case), not the full class of BPS matrices.

Beyond these, alternative algorithmic frameworks include ULV factorizations~\citep{chandrasekaran2003fast}, hierarchically semiseparable (HSS) matrix techniques~\citep{chandrasekaran2007fast, chandrasekaran2006fast, massei2020hm, xia2010fast}, and rational Krylov approaches~\citep{fasino2005rational, vandebril2008rational}. Structure-preserving properties under rotational QR steps have also been extensively explored~\citep{vandebril2005implicit,delvaux2006rank, delvaux2006structures, eidelman2005qr, delvaux2008givens, vandebril2005note}.

Crucially, however, almost all existing $O(n)$ QR factorization schemes for BPS matrices in the rank-structured literature are fundamentally anchored in Givens rotations~\citep{delvaux2008qr, mastronardi2001fast, van2004two}. While rotational approaches are well-developed, they rely on localized transformations that do not yield the standard Householder factor matrix layout native to LAPACK storage formats and modern dense computing pipelines; instead, the resulting orthogonal matrix in the QR factorization is represented as a product of Givens rotations. Furthermore, in order to preserve structure Givens-based approaches typically eliminate subdiagonal elements starting from the bottom-left corner of the matrix, rendering them ill-suited for adaptive precision or dynamically truncated computations.

A special class of BPS matrices with no lower semiseparable part ($r = 0$) are almost banded matrices, which were used by~\cite{olver2013fast} to represent infinite-dimensinoal discretisations of differential equations using the ultraspherical spectral method, accompanied by an optimal complexity adaptive QR factorization with error control and back-substitution solver. To go beyond simple ordinary differential equations with (approximately) polynomial variable coefficients one needs to consider more general BPS matrices. This motivates the question: can we generalize the adaptive QR factorization to general BPS matrices? Because they require beginning the factorization in the bottom left, the existing Givens-based approaches do not lend themselves to the infinite-dimensional linear algebra setting.

In this paper we answer this question in the positive by developing  a Householder-based approach that naturally allows for adaptive, partial QR factorizations. Rather than building a customised  QR factorization, we use the standard QR factorization with Householder reflections, but where the structure at each stage is captured precisely. 
Beyond adaptivity, our Householder-based framework seamlessly interfaces with LAPACK as it represents the same factors matrix as LAPACK in an unblocked QR factorization, but where the factors matrix now has a special memory-efficient storage.

In detail: we prove that the QR factorization of a BPS matrix via Householder reflections yields a full factor matrix $F$ (encoding both $R$ and the reflectors) that is itself BPS. The core of our idea lies in characterizing the precise structural evolution during the QR factorization process. Although the BPS structure is not strictly invariant under individual Householder transformations, we demonstrate that the intermediate matrices remain within a well-defined space of structured perturbations, denoted as $\mathcal{P}(A)$ (Definition~\ref{def:P}). The fundamental insight is that each Householder reflection preserves this extended structure (Lemma~\ref{Lemma:Structure_preserve}), ensuring that the final factor matrix $F$ maintains a compact representation (Theorem~\ref{thm:structure_preserve}). By leveraging this structural inheritance, we design a complete $O(n)$ direct solver—including $Q^\top$ application and backward substitution—implemented in a high-performance, open-source Julia package. Furthermore, we extend this framework to symmetric BPS matrices, proving that the $RQ$ product also retains the BPS structure.

The rest of this paper is organized as follows.
Section~\ref{sec:factors} discusses the factor matrix representation for the QR factorization as computed with Householder reflections, which proves a convenient language to discuss the perturbations of the intermediate stages of the factorization.
Section~\ref{sec:main} presents our main theoretical contributions on the preservation of a specific perturbation structure and the resulting BPS structure of the final QR factorization. Section~\ref{sec:algorithms} details the practical implementation of $O(n)$ algorithms for QR factorization, application of $Q^\top$, and backward substitution. Section~\ref{sec:experiments} presents numerical experiments that confirm the linear complexity, verify the numerical stability, and demonstrate performance advantages of our algorithms, including their compatibility with LAPACK. We conclude in Section~\ref{sec:conclusions} with a discussion of future work. In Appendix~\ref{sec:bps_T} we show that the  WY representation~\citep{schreiber1987storage}, $Q = I - V T V^\top$, also has BPS structure that can be computed in optimal complexity. Finally, Appendix~\ref{sec:fastrq} extends this framework to symmetric BPS matrices, proving that the reverse $RQ$ product also preserves the BPS format.

\begin{remark}
While this paper focuses on real square matrices, the techniques and results naturally extend to complex matrices by replacing transposes with conjugate transposes. The extension to rectangular matrices is also straightforward, as Householder QR factorization applies similarly to rectangular matrices, and the structural arguments carry over with minor adjustments. 
\end{remark}

\begin{remark}
Modern blocked LAPACK routines use the Compact WY representation~\citep{schreiber1987storage}, $Q = I - V T V^\top$, to leverage BLAS-3 operations. When block-accumulating $k$ successive Householder reflections, the upper-triangular matrix $T$ inherently inherits the BPS structure due to the underlying rank structure of the Householder vectors in $V$ (as proved in Appendix~\ref{sec:bps_T}). Crucially, by exploiting the low-rank generators of $V$, each step of LAPACK's \texttt{DLARFT} recursion requires only $O(1)$ operations, allowing us to build the block factor $T \in \mathbb{R}^{k \times k}$ in $O(k)$ complexity, completely independent of $n$. This guarantees that forming the Compact WY representation retains strict $O(n)$ linear complexity with respect to the matrix dimension $n$. Consequently, this algebraic compatibility bridges our structured factorization with high-performance dense BLAS-3 operations without compromising asymptotic efficiency.
\end{remark} 

\begin{remark}
Our approach requires an initial $O(n)$ precomputation in the form of a lookup table of partial inner products, see Remark~\ref{Remark:lookup}. 
This can be compared to the Givens-based approaches which require an initial $O(n)$ cost to eliminate the lower semiseparable part of the matrix. A benefit of the Householder-based approach is that when the semiseparable factors have specific structured forms---for example $U = [1, 1/2, 1/3, \dots, 1/n]^\top$---the precomputation can be done in closed form and the computational complexity of a partial QR factorization then depends solely on the number of columns upper-triangularized, remaining entirely independent of the global system size $n$. This parameter-independent complexity property is particularly vital in  infinite-dimensional spectral theory and resolvent-based eigenvalue analysis, as pioneered by~\cite{colbrook2021computing}. 
\end{remark}

\section{The Representation of the QR Factor Matrix}
\label{sec:factors}

Consider the computation of the QR factorization using Householder reflections which we can write as:
\[
A =\underbrace{ (I - \tau_1 \bfy_1 \bfy_1^\top)\cdots (I - \tau_{n-1} \bfy_{n-1} \bfy_{n-1}^\top)}_Q \underbrace{\begin{bmatrix}
r_{11} & r_{12} & r_{13} & \cdots & r_{1n} \\
 & r_{22} & r_{23} & \cdots & r_{2n} \\
 &  & r_{33} & \cdots & r_{3n} \\
 &  &  & \ddots & \vdots \\
 &  &  &  & r_{nn}
\end{bmatrix}}_R
\]
where
\begin{equation}\label{y_k}
\bfy_k = \begin{bmatrix} 0 \\ \vdots \\ 0 \\ 1 \\ y_{k+1,k} \\ \vdots \\ y_{n,k} \end{bmatrix}
\begin{array}{l} \left.\vphantom{\begin{bmatrix} 0 \\ 0 \\ 1 \end{bmatrix}}\right\} k-1 \text{ zeros} \\ \\ \left.\vphantom{\begin{bmatrix} y_{k+1,k} \\ \vdots \\ y_{n,k} \end{bmatrix}}\right\} n-k \text{ elements}
\end{array}
\end{equation}
and $\tau_k = 2\|\bfy_k\|^{-2}$ so that $\sqrt{\tau_k/2} \bfy_k$ has unit norm.
Following the LAPACK format \citep{doi:10.1137/1.9780898719604}, the QR factorization can be represented by a {\it factor matrix} containing the entries of $\bfy_k$ and $R$:
\begin{equation*}
F = \begin{bNiceMatrix}[cell-space-limits=10pt,columns-width=20pt]
r_{11}      & r_{12} & \cdots       & r_{1k}  \\
y_{21}      & r_{22} & \cdots       & r_{2k}  \\
\vdots      & \ddots & \ddots       & \vdots  \\
y_{k1}      & \cdots & y_{k,k-1}    & r_{kk}  \\
\CodeAfter
  \tikz \draw ([xshift=2pt,yshift=0pt]1.5-|1) -- ([xshift=6pt,yshift=0pt]5-|4.5);
\end{bNiceMatrix}
\end{equation*}
alongside a vector $\bftau = [\tau_1,\ldots,\tau_{n-1},0]^\top$.

We will demonstrate that \( F \) itself retains a BPS structure, which implies structure in $R$ and the Householder reflections. Specifically, $F$ has a lower semiseparable part of rank \( r \), an upper semiseparable part of rank \( r+p \), a lower bandwidth of \( \ell \), and an upper bandwidth of \( \ell+m \).

Intermediate stages of the Householder reflection correspond to a {\it partial factor matrix}. In particular, if we apply $k$ Householder reflections to partially upper triangularize $A$ as
\[
(I - \tau_k \bfy_k \bfy_k^\top) \cdots (I - \tau_1 \bfy_1 \bfy_1^\top) A =
\begin{bNiceMatrix}[cell-space-limits=5pt]
r_{11} & r_{12} & \cdots & r_{1k} & r_{1,k+1} & \cdots & r_{1n} \\
       & r_{22} & \cdots & r_{2k} & r_{2,k+1} & \cdots & r_{2n} \\
       &        & \ddots & \vdots & \vdots    & \ddots & \vdots \\
       &        &        & r_{kk} & r_{k,k+1} & \cdots & r_{kn} \\
       &        &        &        & a_{k+1,k+1}^k & \cdots & a_{k+1,n}^k \\
       &        &        &        & \vdots        & \ddots & \vdots \\
       &        &        &        & a_{n,k+1}^k   & \cdots & a_{nn}^k
\CodeAfter
  \tikz \draw ([xshift=2pt,yshift=-2pt]1.5-|1) -- ([xshift=2pt,yshift=-2pt]5-|4.5);
    \tikz \draw ([xshift=2pt,yshift=-2pt]5-|4.5) -- ([xshift=-2pt,yshift=-2pt]5-|8);
        \tikz \draw ([xshift=0pt,yshift=-2pt]5-|5) -- (8-|5);
\end{bNiceMatrix}.
\]
We can encode this partial factorization in a partial factor matrix:
\[
F_k = \begin{bNiceMatrix}[cell-space-limits=10pt,columns-width=20pt]
r_{11}      & r_{12} & \cdots       & r_{1k} & r_{1,k+1} & \cdots & r_{1n} \\
y_{21}      & r_{22} & \cdots       & r_{2k} & r_{2,k+1} & \cdots & r_{2n} \\
\vdots      & \ddots & \ddots       & \vdots & \vdots    & \ddots & \vdots \\
y_{k1}      & \cdots & y_{k,k-1}    & r_{kk} & r_{k,k+1} & \cdots & r_{kn} \\
y_{k+1,1}   & \cdots &  y_{k+1,k-1} &   y_{k+1,k}     & a_{k+1,k+1}^k & \cdots & a_{k+1,n}^k \\
\vdots      & \ddots &    \vdots    &  \vdots      & \vdots        & \ddots & \vdots \\
y_{n1}      & \cdots &    y_{n,k-1} &  y_{nk}      & a_{n,k+1}^k   & \cdots & a_{nn}^k
\CodeAfter
  \tikz \draw ([xshift=2pt,yshift=2pt]1.5-|1) -- ([xshift=8pt,yshift=0pt]5-|4.5);
    \tikz \draw ([xshift=8pt,yshift=0pt]5-|4.5) -- ([xshift=6pt,yshift=0pt]5-|8.5);
        \tikz \draw ([xshift=6pt,yshift=0pt]5-|5) -- ([xshift=6pt,yshift=0pt]8-|5);
\end{bNiceMatrix},
\]
alongside $\bftau_k = [\tau_1,\ldots,\tau_k,0,\ldots,0]^\top$, for $k = 1, \dots, n-1$.

Specifically, we define $F_0 := A$ as the initial state. Although the upper triangularization is computationally complete after $n-1$ steps, we formally define $F_n:=F$ as the terminal state where the residual perturbation in the bottom-right $1 \times 1$ block of $F_{n-1}$—which at this stage is still implicitly represented within the structured perturbation format (to be explicitly defined in Section \ref{sec:main})—is formally absorbed into the final diagonal entry $r_{nn}$. Consequently, the upper triangular part of $F_n$ constitutes the full factor $R$, while its strictly lower triangular part stores the Householder vectors $\{\mathbf{y}_k\}_{k=1}^{n-1}$.


The intermediate factor matrices {\it do not} have BPS structure, however, we shall show that it is a structured perturbation of BPS.

\section{Structure of the factor matrix for BPS matrices}
\label{sec:main}

We now turn our attention to understanding the structure that $F_k$ exhibits.  In particular, we claim it is a structured perturbation of a BPS matrix. To be precise, we first introduce a notation for BPS matrices:

\begin{definition}
Denote the set of {\it banded matrices}  with lower-bandwidth $\ell$, and upper-bandwidth $m$  as $\BM \subset \bbR^{n \times n}$. In particular, $B \in \BM \subset \bbR^{n \times n}$ if
 \( B = (b_{kj})_{k,j=1}^n \in \mathbb{R}^{n \times n} \) is a banded matrix satisfying \( b_{kj}=0 \) for \( k-j>\ell \) or \( j-k>m \).
\end{definition}

\begin{definition}
Denote the set of {\it banded-plus-semiseparable matrices} which have lower-semiseparable rank $r$, upper-semiseparable rank $p$, lower-bandwidth $\ell$, and upper-bandwidth $m$  as $\BPS \subset \bbR^{n \times n}$. In particular, $A \in \BPS \subset \bbR^{n \times n}$ if
\[
A = B + \tril(UV^\top, -1) + \triu(WS^\top, 1)\label{express_A}
\]
for \( U, V \in \mathbb{R}^{n\times r} \), \( W, S \in \mathbb{R}^{n\times p} \), and \( B   \in \BM \).
\end{definition}

We note that the BPS structure is not uniquely defined, so we implicitly assume that if we know $A \in \BPS$ we have access to $U,V,W,S$. This structure is preserved under principal subsection:
\begin{align*}
A[k\!:\!n, k\!:\!n] &= B[k\!:\!n, k\!:\!n]
+
\tril(U[k\!:\!n, :]V[k\!:\!n, :]^\top,-1) \\
&\qquad +
\triu(W[k\!:\!n, :]S[k\!:\!n, :]^\top,1) \in {\mathrm{BPS}}_{(r,p),(\ell,m)}^{n-k+1 \times n-k+1}.
\end{align*}

We will also need to work with padded vectors or matrices:

\begin{definition}
Denote a {\it padded vector} with only the first $p$ rows being possibly nonzero as $\calS_p^n \subset \bbR^n$ and a {\it padded matrix} with only the principal $p \times q$ subblock being possibly nonzero as $\calS_{p \times q}^{m \times n} \subset \bbR^{m \times n}$. Explicitly, $\bfx \in \calS_p^n$ if $\bfx[k] = 0$ unless $k \leq p$ and $A \in \calS_{p \times q}^{m \times n}$ if $A[k,j] = 0$ unless $k \leq p$ and $j \leq q$.
\end{definition}

Note for example if $B \in \BM$ then $B[k\!:\!n,k] \in \calS_{\ell+1}^{n-k+1}$ and $B[k,k\!:\!n] \in \calS_{m+1}^{n-k+1}$.

We will now describe the structure of each portion of the partial factor matrix. We will describe the bottom right in terms of the following linear space of matrices:

\begin{definition} \label{def:P}
Given
\[
A = B + \tril(UV^\top, -1) + \triu(WS^\top, 1)  \in \BPS
\]
define the linear space:
\begin{align*}
\calP(A) &:= \bigg\{
UJS^\top + UKU^\top A + UE + XS^\top + YU^\top A + Z :\\
&\qquad\quad  J \in \mathbb{R}^{r\times p}, K \in \mathbb{R}^{r\times r},  E \in \calS_{r\times (\ell+m)}^{r \times n},  X \in \calS_{\ell \times p}^{n \times p}, \\
 &\qquad\quad  Y \in \calS_{\ell \times r}^{n \times r}, Z \in \calS_{\ell \times (\ell+m)}^{n \times n}
\bigg\} \subset \bbR^{n \times n}.
\end{align*}
\end{definition}

For convenience we use subscript $s$ to define the nonzero portions of these matrices as $E_s:=E[:,1:\min(\ell+m,n)]$, $X_s:=X[1:\min(\ell,n),:]$, $Y_s:=Y[1:\min(\ell,n),:]$, and $Z_s:=Z[1:\min(\ell,n),1:\min(\ell+m,n)]$.

We will show that $F_k$ has the following structure:
The first $k$ rows/columns of $F_k$ will maintain the BPS structure albeit with increased upper semiseparable rank of $p+r$ and upper bandwidth of $\ell+m$. The bottom right will be a structured perturbation of a BPS matrix.

In particular,
the partial factor matrices will satisfy the following:

\begin{definition}
We say $F_k \in {\rm BPSF}_k(A)$ if for
\[
\tilde S :=  \begin{bmatrix} S & A^\top U
\end{bmatrix} \in \bbR^{n \times (r+p)},
\]
 there exist 
 $B_k \in {\mathrm{BM}}_{(\ell,\ell+m)}^{n \times n}$,
$V_k \in \bbR^{n \times r}$, $W_k \in \bbR^{n \times (r+p)}$, 
so that the first $k$ rows and columns satisfy the conditions of the BPS format:
\begin{align*}
F_k[:,1\!:\!k] &= (B_k + \tril(U V_k^\top, -1) + \triu(W_k \tilde S^\top,1))[:,1\!:\!k] \\
F_k[1\!:\!k,:] &= (B_k + \tril(U V_k^\top, -1) + \triu(W_k \tilde S^\top,1))[1\!:\!k,:]
\end{align*}
whilst the bottom right block is a perturbed BPS matrix:
\begin{align*}
F_k[k\!+\!1\!:\!n,k\!+\!1\!:\!n] &= A[k+1\!:\!n,k+1\!:\!n] + P_k,
\end{align*}
where $P_k \in \calP(A[k+1\!:\!n,k+1\!:\!n])$.
\end{definition}

We will show that the factor matrix $F_k \in {\rm BPSF}_k(A)$ by induction.
As a preliminary step we note that the perturbation structure of the bottom right is preserved under truncation:

\begin{lemma}\label{Lemma:Ppertub}
    If $P \in \calP(A)$ then $P[k\!:\!n,k\!:\!n] \in \calP(A[k\!:\!n,k\!:\!n])$.
\end{lemma}
\begin{proof}
    We show for $k=2$ with the other cases following by induction. We write:
\begin{align*}
    P[2\!:\!n,2\!:\!n] &= U[2\!:\!n,:] J S[2\!:\!n,:]^\top + 
    U[2\!:\!n,:] K \underbrace{U^\top A[:,2\!:\!n]}_{=U[2\!:\!n,:]^\top A[2\!:\!n,2\!:\!n] + U[1,:] A[1,2\!:\!n]^\top}\\
    &\qquad +
    U[2\!:\!n,:] E[:,2\!:\!n] + X[2\!:\!n,:] S[2\!:\!n,:]^\top \\
    \\
    &\qquad + Y[2\!:\!n,:]U^\top A[:,2\!:\!n] + Z[2\!:\!n,2\!:\!n].
\end{align*}
We can then write
\begin{align*}
U[2\!:\!n,:]K U[1,:] A[1,2\!:\!n]^\top &= U[2\!:\!n,:] \underbrace{K U[1,:] B[1,2\!:\!n]^\top}_{\in \calS_{r \times m}^{r \times (n-1)}} \\
&\qquad + U[2\!:\!n,:] \underbrace{K U[1,:]W[1,:]^\top}_{\in \bbR^{r \times p}} S[2\!:\!n,:]^\top.
\end{align*}
Similarly,
\begin{align*}
Y[2\!:\!n,:] U[1,:] A[1,2\!:\!n]^\top &= \underbrace{Y[2\!:\!n,:] U[1,:] B[1,2\!:\!n]^\top}_{\in \calS_{(\ell-1) \times m}^{(n-1) \times (n-1)}} \\
&\qquad +  \underbrace{Y[2\!:\!n,:] U[1,:]W[1,:]^\top}_{\in \calS_{(\ell-1) \times p}^{(n-1) \times p}} S[2\!:\!n,:]^\top.
\end{align*}
The lemma follows since we have
\begin{align*}
P[2\!:\!n,2\!:\!n] &= U[2\!:\!n,:] (J + K U[1,:] W[1,:]^\top) S[2\!:\!n,:]^\top + \\
&\qquad U[2\!:\!n,:] K U[2\!:\!n,:]^\top A[2\!:\!n,2\!:\!n] \\
&\qquad + U[2\!:\!n,:](E[:,2\!:\!n] + K U[1,:] B[1,2\!:\!n]^\top) \\
&\qquad + (X[2\!:\!n,:] + Y[2\!:\!n,:]U[1,:]W[1,:]^\top)S[2\!:\!n,:]^\top \\
&\qquad + Y[2\!:\!n,:]U[2\!:\!n,:]^\top A[2\!:\!n,2\!:\!n] \\
&\qquad + Z[2\!:\!n,2\!:\!n] + Y[2\!:\!n,:] U[1,:]B[1,2\!:\!n]^\top,
\end{align*}
which is precisely the required form.
\end{proof}

The next step is to understand the structure of the Householder reflector, which corresponds to rows $k+1$ through $n$ of
the $k$-th column of $F_k$:

\begin{lemma} \label{lemma:lower}
    For $1\leq k \leq n-1$, suppose $F_{k-1}  \in {\rm BPSF}_{k-1}(A)$. Then there exists $B_k$ and $V_k$ such that
\[
F_k[:,1\!:\!k] = (B_k + \tril(U V_k^\top, -1) + \triu(W_{k-1} \tilde S^\top,1))[:,1\!:\!k]
\]
\end{lemma}
\begin{proof}
The first $k-1$ columns follow by defining:
\begin{align*}
B_k[:,1\!:\!k\!-\!1] &:= B_{k-1}[:,1\!:\!k\!-\!1] \\
V_k[1\!:\!k\!-\!1,:] &:= V_{k-1}[1\!:\!k\!-\!1,:].
\end{align*}
The first $k$ rows of the $k$-th column follow by defining $B_k[k,k] := r_{kk}$. Thus the main contribution of this lemma is structure in the $k$-th Householder reflector which corresponds to the $k+1$ through $n$th rows.

Since $F_{k-1}  \in {\rm BPSF}_{k-1}(A)$ we may write
\begin{align*}
    F_{k-1}[k:n,k:n] = A[k:n,k:n] + P_{k-1}
\end{align*}

where
\begin{align*}
 P_{k-1} := & U[k:n,:]J_kS[k:n,:]^\top + U[k:n,:]K_kU[k:n,:]^\top A[k:n,k:n]\\&+U[k:n,:]E_k+X_kS[k:n,:]^\top + Y_kU[k:n,:]^\top A[k:n,k:n] + Z_k
\end{align*}
with $J_k\in \mathbb R^{r\times p}$, $K_k\in \mathbb R^{r\times r}$, $E_k\in \calS^{r\times (n-k+1)}_{r\times (\ell+m)}$, $X_k\in \calS^{(n-k+1)\times p}_{\ell\times p}$, $Y_k\in \calS^{(n-k+1)\times r}_{\ell \times r}$, and $Z_k \in \calS^{(n-k+1)\times(n-k+1)}_{\ell \times(\ell+m)}$.

The $k$-th Householder reflection is defined in terms of
\begin{align}
F_{k-1}[k:n,k] &= A[k:n,k] +   P_{k-1}[1:n-k+1, 1] \nonumber\\
&= B[k\!:\!n,k] + \begin{bmatrix}
0\\
    U[k+1\!:\!n,:] 
\end{bmatrix} V[k,:] \nonumber\\
& \qquad +  \begin{bmatrix}
0 \\
    U[k+1\!:\!n,:] 
\end{bmatrix} \underbrace{(J_k S[k,:] + K_k U^\top A[:,k] + E_k[:,1])}_{=: \bfv_k \in \bbR^r}
\nonumber\\ \label{eq:tilde_b_k}
&\qquad + \underbrace{\begin{bmatrix}
    X_{ks} S[1,:] + Y_{ks} U^\top A[:,k] +Z_{ks}[:,1] \in \bbR^{\min(\ell,n-k+1)} \\ \bfzero
\end{bmatrix}}_{=: \tilde \bfb_k}.
\end{align}

Letting
\begin{align}\label{eq:alpha_k}
\alpha_k := \text{sign}(F_{k-1}[k,k])\big| F_{k-1}[k,k] +\text{sign}(F_{k-1}[k,k])\| F_{k-1}[k\!:\!n,k]\|  \big|,
\end{align}

we define:
\begin{align}\label{eq:hv_update_1}
B_k[k+1\!:\!n,k] &:= {B[k+1:n,k] + \tilde \bfb_k[2:n-k+1] \over \alpha_k}
\end{align}
and
\begin{align}\label{eq:hv_update_2}
V_k[k,:] &:= {\bfv_k + V[k,:] \over \alpha_k}.
\end{align}
The remaining (not structurally zero) entries of $B_k$ and $V_k$ are at this point arbitrary. 
\end{proof}

We now show the first $k$ rows of $F_k$ (corresponding to $R$) also have the specified form:

\begin{lemma} \label{lemma:upper}
    For $1\leq k \leq n-1$, suppose $F_{k-1}  \in {\rm BPSF}_{k-1}(A)$. Then there exists $B_k, V_k, W_k$ such that
\[
F_k[1:k,:] = (B_k + \tril(U V_k^\top, -1) + \triu(W_k \tilde S^\top,1))[1:k,:]
\]
\end{lemma}
\begin{proof}
The first $k-1$ rows follow by defining $V_k$ as before and:
\begin{align*}
B_k[1\!:\!k-1,:] &:= B_{k-1}[1\!:\!k-1,:], \\
W_k[1\!:\!k-1,:] &:= W_{k-1}[1\!:\!k-1,:].
\end{align*}

We now consider the first row in applying the $k$-th Householder reflection to
\begin{align*}
A_{k-1} &:= F_{k-1}[k\!:\!n,k\!:\!n] \\
&= A[k\!:\!n,k\!:\!n] +
\underbrace{(U[k\!:\!n,:]  J_k+ X_k[1\!:\!n-k+1,:])}_{=:\hat T_1 \in \bbR^{(n-k+1) \times p}} S[k\!:\!n,:]^\top  \\
&\qquad + \underbrace{(U[k\!:\!n,:]  K_k + Y_k[1\!:\!n-k+1,:])}_{=:T_2\in \bbR^{(n-k+1) \times r}} U[k:n,:]^\top A[k:n,k:n] \\
&\qquad + \underbrace{U[k\!:\!n,:]  E_k[:,1\!:\!n-k+1] + Z_k[1\!:\!n-k+1,1\!:\!n-k+1]}_{=:\hat T_3 \in \calS_{(n-k+1) \times (\ell +m)}^{(n-k+1) \times (n-k+1)}}.
\end{align*}
Since
\begin{equation}
\begin{aligned}\label{eq:UtA_relation}
U[k\!:\!n,:]^\top A[k\!:\!n,k\!:\!n] &= (A^\top U)[k:n,:]^\top-\underbrace{(\sum_{t=1}^{k-1}U[t,:]W[t,:]^\top)}_{=:P_1\in \mathbb{R}^{r\times p}}S[k:n,:]^\top \\& -\underbrace{\sum_{t=\max(k-m,1)}^{k-1}U[t,:]B[t,k:n]}_{=:P_2\in \calS_{r\times m}^{r\times (n-k+1)}},
\end{aligned}
\end{equation}

we can also write $A_{k-1}$ as
\begin{align*}
A_{k-1} 
&= A[k\!:\!n,k\!:\!n] +
\underbrace{(\hat T_1-T_2P_1)}_{=:T_1 \in \bbR^{(n-k+1) \times p}} S[k\!:\!n,:]^\top  \\
&\qquad + T_2 (U^\top A)[:,k\!:\!n] + \underbrace{\hat T_3 - T_2P_2}_{=:T_3 \in \calS_{(n-k+1) \times (\ell +m)}^{(n-k+1) \times (n-k+1)}}.
\end{align*}

Noting that
\[
 A[k,k\!:\!n]^\top = 
B[k,k\!:\!n]^\top + \underbrace{\bfe_k^\top \triu(W S^\top,1)[:,k\!:\!n]}_{ 
    W[k,:]^\top S[k:n,:]^\top   - W[k,:]^\top S[k,:] \bfe_1^\top
},
\]
we can write the first row as
\begin{align*}
\bfe_1^\top A_{k-1} &= \bfs_1^\top  + \underbrace{(W[k,:]^\top + \bfe_1^\top T_1)}_{=:\bfw_1^\top \in \bbR^{1 \times p}} S[k\!:\!n,:]^\top  + \underbrace{\bfe_1^\top T_2}_{=:\tilde\bfw_1^\top \in \bbR^{1 \times r}} (A^\top U)[k\!:\!n,:]^\top
\end{align*}
for the sparse row vector (which will correspond to a slice of a banded matrix)
\begin{align*}
\bfs_1^\top  :=
    \underbrace{B[k,k\!:\!n]^\top}_{\in \calS_{1 \times(m+1)}^{1\times(n-k+1)}}  - \underbrace{W[k,:]^\top S[k,:] \bfe_1^\top}_{\in \calS_{1 \times 1}^{1\times(n-k+1)}} + \underbrace{\bfe_1^\top T_3 }_{\in \calS_{1 \times (\ell+m)}^{1 \times (n-k+1)}} \in \calS_{1 \times (\ell+m+1)}^{1 \times (n-k+1)}.
\end{align*}
For the vector associated with the $k$-th Householder reflection
\begin{equation}
\begin{aligned} \label{eq:y_k}
\bfy_k &:= \begin{bmatrix}1 \\
    F_k[k+1\!:\!n,k]
\end{bmatrix} \\
&=\underbrace{B_k[k\!:\!n,k]+(1-B_k[k,k] - U[k,:]^\top V_k[k,:]) \bfe_1}_{=: \hat \bfb_k \in \calS_{\ell+1}^{n-k+1}} + U[k\!:\!n,:] V_k[k,:],
\end{aligned}
\end{equation}
we first write
\begin{align} \label{eq:yA}
\bfy_k^\top A[k\!:\!n,k\!:\!n] &=  \bfs_2^\top  + V_k[k,:]^\top U[k\!:\!n,:]^\top A[k\!:\!n,k\!:\!n]   +  \hat \bfb_k^\top W[k\!:\!n,:]  S[k\!:\!n,:]^\top
\end{align}
for 
\begin{align*}
\bfs_2^\top &:= \underbrace{\hat \bfb_k^\top B[k\!:\!n,k\!:\!n]}_{\in \calS_{1 \times(\ell+m+1)}^{1\times(n-k+1)}} + \underbrace{\hat \bfb_k^\top \tril(U[k\!:\!n,:]  V[k\!:\!n,:]^\top,-1)}_{\in \calS_{1 \times \ell+1}^{1\times(n-k+1)}}
\\& \qquad - \underbrace{\hat \bfb_k^\top \tril(W[k\!:\!n,:]  S[k\!:\!n,:]^\top,0)}_{\in \calS_{1 \times (\ell+1)}^{1\times(n-k+1)}} \in  \calS_{1 \times(\ell+m+1)}^{1\times(n-k+1)}.
\end{align*}

We therefore find
\begin{align*}
\bfy_k^\top A_{k-1} &= 
\underbrace{\bfs_2^\top - V_k[k,:]^\top P_2+  \bfy_k^\top T_3}_{=: \bfs_3^\top \in \calS_{1 \times \ell+m+1}^{1 \times (n-k+1)}} + \underbrace{(V_k[k,:]^\top  + \bfy_k^\top T_2)}_{=: \tilde \bfw_2^\top \in \bbR^{1 \times r}} (A^\top U)[k\!:\!n,:]^\top\\
&\qquad + 
  \underbrace{(\hat \bfb_k^\top W[k\!:\!n,:] - V_k[k,:]^\top P_1 + \bfy_k^\top T_1)}_{=: \bfw_2^\top \in \bbR^{1 \times p}}  S[k\!:\!n,:]^\top.
\end{align*}
Thus we find for $\beta_k:= -\tau_k \bfe_1^\top \bfy_k$ that applying the Householder reflection yields
\begin{align*}
    \bfe_1^\top (I - \tau_k \bfy_k \bfy_k^\top)A_{k-1} &=
    \bfs_1^\top  + \beta_k \bfs_3^\top + (\bfw_1^\top + \beta_k \bfw_2^\top) S[k\!:\!n,:]^\top \\
    &\qquad  + (\tilde \bfw_1^\top + \beta_k \tilde \bfw_2^\top) (A^\top U)[k\!:\!n,:]^\top
\end{align*}
and hence the results from defining
\begin{align}\label{eq:row_updtae_1}
W_k[k,:]^\top &:= \begin{bmatrix} 
    (\bfw_1^\top + \beta_k \bfw_2^\top) & (\tilde \bfw_1^\top + \beta_k \tilde \bfw_2^\top) \end{bmatrix} \in \bbR^{1 \times (r+p)}
\end{align}
and 
\begin{align}\label{eq:row_updtae_2}
B_k[k,k\!:\!n]^\top &:= \bfs_1^\top  + \beta_k \bfs_3^\top  \in \calS_{1 \times (\ell+m+1)}^{1 \times (n-k+1)}.
\end{align}

\end{proof}

We put these results together to show that the structure for partial factor matrices is preserved at each iteration:

\begin{lemma} \label{Lemma:Structure_preserve}
    For $1\leq k \leq n-1$, suppose $F_{k-1}  \in {\rm BPSF}_{k-1}(A)$. Then $F_k \in {\rm BPSF}_k(A)$. 
\end{lemma}

\begin{proof}
    We have already showed the result apart from the bottom right, in particular we need to show
\begin{align*}
A_k &:= F_k[(k\!+\!1)\!:\!n,(k\!+\!1)\!:\!n] \\
&= 
 ((I - \tau_k \bfy_k \bfy_k^\top)A_{k-1})[2\!:\!n-k+1,2\!:\!n-k+1] \\
 &\in A[(k+1)\!:\!n,(k+1)\!:\!n]+\calP(A[(k+1)\!:\!n,(k+1)\!:\!n]).
\end{align*}
From Lemma~\ref{Lemma:Ppertub} we know that since $A_{k-1} \in A[k\!:\!n,k\!:\!n]+\calP(A[k\!:\!n,k\!:\!n])$ we have
\[
A_{k-1}[2\!:\!(n-k+1),2\!:\!(n-k+1)] \in A[(k\!+\!1)\!:\!n,(k\!+\!1)\!:\!n]+\calP(A[(k\!+\!1)\!:\!n,(k\!+\!1)\!:\!n]).
\]

By applying the same technique as in the proof of Lemma~\ref{Lemma:Ppertub}, we can express
\begin{align*}
    & A_{k-1}[2\!:\!(n-k+1),2\!:\!(n-k+1)] \\=&  A[(k\!+\!1)\!:\!n,(k\!+\!1)\!:\!n] + U[(k\!+\!1)\!:\!n,:]J^{(k)}S[(k\!+\!1)\!:\!n,:]^\top \\&+U[(k\!+\!1)\!:\!n,:]K^{(k)}U[(k\!+\!1)\!:\!n,:]^\top A[(k\!+\!1)\!:\!n,(k\!+\!1)\!:\!n] \\&
    +U[(k\!+\!1)\!:\!n,:]E^{(k)} + X^{(k)}S[(k\!+\!1)\!:\!n,:]^\top \\&
    + Y^{(k)}U[(k\!+\!1)\!:\!n,:]^\top A[(k\!+\!1)\!:\!n,(k\!+\!1)\!:\!n] + Z^{(k)}
\end{align*}
with $J^{(k)}\in \mathbb R^{r\times p}$, $K^{(k)}\in \mathbb R^{r\times r}$, $E^{(k)}\in \calS_{r\times(\ell+m)}^{r\times (n-k)}$, $X^{(k)}\in \calS_{\ell \times p}^{(n-k)\times p}$, $Y^{(k)}\in \calS_{\ell \times r}^{(n-k)\times r}$ and $Z^{(k)}\in \calS_{\ell \times(k+m)}^{(n-k)\times (n-k)}$.

We thus need only consider $\bfy_k \bfy_k^\top A_{k-1}$.
Let $\hat U_k\in \mathbb R^{n-k+1}$ s.t. $\hat U_k[1,:] = \mathbf{0}$ and $\hat U_k[2:n-k+1,:]=U[k+1:n,:]$; $\hat A_k \in \mathbb R^{(n-k+1)\times(n-k+1)}$ s.t. $\hat A_k[1,:] = \mathbf{0}$ and $\hat A_k[2:n-k+1,:] = A[k+1:n,:]$.

Using the technique from the proof of Lemma \ref{lemma:upper} we can write

\begin{align*}
U[k\!:\!n,:]^\top A[k\!:\!n,k\!:\!n] &= \hat U_k^\top \hat A_k+\underbrace{U[k,:]W[k,:]^\top}_{=:P_3\in \mathbb{R}^{r\times p}}S[k:n,:]^\top \\& + \underbrace{U[k,:]B[k,k:n]-U[k,:]W[k,:]^\top S[k,:]\bfe_1^\top}_{=:P_4\in \calS_{r\times m}^{r\times (n-k+1)}}
\end{align*}

and

\begin{align*}
\bfy_k^\top A_{k-1} &= 
\underbrace{\bfs_2^\top +V_k[k,:]^\top P_4+  \bfy_k^\top T_3}_{=: \bfs_4^\top \in \calS_{1 \times \ell+m+1}^{1 \times (n-k+1)}} + \underbrace{(V_k[k,:]^\top  + \bfy_k^\top T_2)}_{=: \tilde \bfw_2^\top \in \bbR^{1 \times r}} \hat U_k^\top \hat A_k\\
&\qquad + 
  \underbrace{(\hat \bfb_k^\top W[k\!:\!n,:] +V_k[k,:]^\top P_3+ \bfy_k^\top T_1)}_{=: \bfw_3^\top \in \bbR^{1 \times p}}  S[k\!:\!n,:]^\top.
\end{align*}

Therefore, 
\begin{align*}
\bfy_k \bfy_k^\top A_{k-1} &=
(\hat \bfb_k + U[k:n,:] V_k[k,:]) \times \\
&\qquad
(\bfs_4^\top + \tilde \bfw_2^\top \hat U_k^\top \hat A_k +
 \bfw_3^\top S[k:n,:]^\top ) \\
 &= \underbrace{\hat \bfb_k \bfs_4^\top}_{\in \calS_{(\ell+1) \times (\ell+m+1)}^{(n-k+1) \times (n-k+1)}}  + \underbrace{\hat \bfb_k \tilde \bfw_2^\top}_{\in \calS_{(\ell+1) \times r}^{(n-k+1) \times r}} \hat U_k^\top \hat A_k + 
 \underbrace{\hat \bfb_k \bfw_3^\top}_{\in \calS_{(\ell+1) \times p}^{(n-k+1) \times p}} S[k:n,:]^\top   \\
&\qquad + U[k:n,:] \underbrace{ V_k[k,:] \bfs_4^\top}_{\in \calS_{r \times (\ell+m+1)}^{r \times (n-k+1)}} + 
U[k:n,:] \underbrace{ V_k[k,:] \tilde \bfw_2^\top}_{\in \bbR^{r \times r}} \hat U_k^\top \hat A_k \\
&\qquad + U[k:n,:] \underbrace{ V_k[k,:] \bfw_3^\top}_{\in \bbR^{r \times p}} S[k:n,:]^\top.
\end{align*}

Noticing that $(\hat U_k^\top \hat A_k)[:,2:n-k+1]= U[k+1:n,:]^\top A[k+1:n,k+1:n]$, we can further write 
\begin{align*}
    A_k=&A_{k-1}[2:n-k+1,2:n-k+1]-(\tau_k \bfy_k \bfy_k^\top A_{k-1})[2:n-k+1,2:n-k+1] \\
    =& A[(k\!+\!1)\!:\!n,(k\!+\!1)\!:\!n] + U[(k\!+\!1)\!:\!n,:]J_{k+1}S[(k\!+\!1)\!:\!n,:]^\top \\&+U[(k\!+\!1)\!:\!n,:] K_{k+1}U[(k\!+\!1)\!:\!n,:]^\top A[(k\!+\!1)\!:\!n,(k\!+\!1)\!:\!n] \\&
    +U[(k\!+\!1)\!:\!n,:] E_{k+1} + X_{k+1}S[(k\!+\!1)\!:\!n,:]^\top \\&
    + Y_{k+1}U[(k\!+\!1)\!:\!n,:]^\top A[(k\!+\!1)\!:\!n,(k\!+\!1)\!:\!n] + Z_{k+1} \\
    \in & A[(k+1)\!:\!n,(k+1)\!:\!n]+\calP(A[(k+1)\!:\!n,(k+1)\!:\!n])
\end{align*}
where
\begin{align}
    J_{k+1} &:= J^{(k)} - \tau_kV_k[k,:] \bfw_3^\top \in \mathbb R^{r\times p}, \label{eq:submatrix_update_J} \\
    K_{k+1} &:= K^{(k)} - \tau_kV_k[k,:]\tilde \bfw_2^\top \in \mathbb R^{r\times r}, \label{eq:submatrix_update_K} \\
    E_{k+1} &:= E^{(k)} - \tau_k V_k[k,:] \bfs_4^\top[2:n-k+1] \in \calS_{r\times(\ell+m)}^{r\times(n-k)},\label{eq:submatrix_update_E} \\
    X_{k+1} &:= X^{(k)} - \tau_k \hat \bfb_k[2:n-k+1]\bfw_3^\top \in \calS_{\ell \times p}^{(n-k)\times p}, \label{eq:submatrix_update_X} \\
    Y_{k+1} &:= Y^{(k)} - \tau_k\hat \bfb_k[2:n-k+1] \tilde \bfw_2^\top \in \calS_{\ell \times r}^{(n-k) \times r}  \label{eq:submatrix_update_Y}, \\
    Z_{k+1} &:= Z^{(k)} - \tau_k \hat \bfb_k[2:n-k+1]\bfs_4^\top [2:n-k+1] \in \calS_{\ell \times (\ell+m)}^{(n-k)\times(n-k)}. \label{eq:submatrix_update_Z}
\end{align}

\end{proof}

We now arrive at the main result: the final factor matrix is BPS:

\begin{mytheorem}\label{thm:structure_preserve}
    The QR factorization of a BPS matrix yields a BPS factor matrix:
$
F_n \in {\mathrm{BPS}}_{(r,r+p),(\ell,\ell+m)}^{n \times n}.
$

\end{mytheorem}

\begin{proof}
    It is easy to see that $F_0:=A\in {\rm BPSF}_{0}(A)$ (with the parameter matrices $J, K, E, X, Y$ and $Z$ all $\mathbf{0}$). From Lemma \ref{lemma:lower}, \ref{lemma:upper}, and \ref{Lemma:Structure_preserve}, an induction argument shows that $F_{n-1}$ can be written as 
\[
F_{n-1} = B_{n-1} + \tril(U V_{n-1}^\top,-1) + \triu(W_{n-1} \tilde S^\top,1)
\]
except at the entry $F_n[n,n]$.

Now define $V_n:=V_{n-1}$, $W_n:=W_{n-1}$, and let $B_n$ coincide with $B_{n-1}$ everywhere except for the last diagonal entry $B_n[n,n]=F_{n-1}[n,n]$. We then obain

\[
F_{n} = B_{n} + \tril(U V_{n}^\top,-1) + \triu(W_{n} \tilde S^\top,1).
\]

\end{proof}

\section{Fast Algorithms for QR Factorization and Direct Solvers}
\label{sec:algorithms}

We now discuss the practical implementation of an $O(n)$ algorithm for computing the QR factorization of a BPS matrix based on theorem \ref{thm:structure_preserve}.

\begin{algorithm}[H]
\caption{Fast QR Factorization for BPS Matrices}
\label{algo:fastqr}
\begin{algorithmic}
\State \textbf{Input} Banded-plus-semiseparable matrix $A \in \BPS$.
\State \textbf{Output} Factor matrix $F \in {\mathrm{BPS}}_{(r,r+p),(\ell,\ell+m)}^{n \times n}$ and scaling vector
$\boldsymbol{\tau} \in \mathbb{R}^{n}$.

\State \textbf{Initialization:}
\State Compute $A^\top U$ in $O(n)$, see \cite{chandrasekaran2002fast}, and set $\tilde S \gets \begin{bmatrix} S & A^\top U
\end{bmatrix}$ 
\State Initialize $B_n \gets \mathbf{0}_{n \times n}$, $V_n \gets \mathbf{0}_{n \times r}$, $W_n \gets \mathbf{0}_{n \times (r+p)}$
\State Initialize parameter matrices: $J \gets \mathbf{0}_{r \times p}$, $K \gets \mathbf{0}_{r \times r}$, $E_s \gets \mathbf{0}_{r \times \min(\ell+m,n)}$, $X_s \gets \mathbf{0}_{\min(\ell,n) \times p}$, $Y_s \gets \mathbf{0}_{\min(\ell,n) \times r}$, $Z_s \gets \mathbf{0}_{\min(\ell,n) \times \min(\ell+m,n)}$

\For{$k = 1$ to $n-1$}
    \State \textbf{Step 1: Form Householder vector $\bfy_k$}
    \State Run $ \text{FormHouseholderVector}(A_{k-1}, U[k:n, :], V[k:n, :], \ell)$ (Algorithm \ref{algo:form_householder}) to get:
    \State $\bar{\mathbf{k}}_{k} \gets \text{low-rank part from } \mathbf{y}_k$
    \State $\mathbf{b}_{k} \gets \text{banded part from } \bfy_k$
    \State $\tau_k \gets \text{scaling coefficient}$
    \State $\boldsymbol{\tau}[k] \gets \tau_k$

    \State \textbf{Step 2: Store $k$-th column of $F$}
    \State $V_n[k, :] \gets \bar{\mathbf{k}}_{k}$ (Store low-rank generator)
    \For{$j = k+1$ to $\min(k+\ell, n)$}
        \State $B_n[j, k] \gets \mathbf{b}_{k}[j-k+1]$ (Store banded part)
    \EndFor

    \State \textbf{Step 3: Compute and store $k$-th row of $F$}
    \State $[\tilde \bfw_k, \tilde{\mathbf{d}}_k] \gets \text{ComputeRowUpdate}(A_{k-1}, U, \tilde S, \mathbf{k}_{k}, \mathbf{b}_{k}, \tau_k)$ (Algorithm \ref{algo:compute_row})
    \State $W_n[k, :] \gets \tilde \bfw_k$
    
    \For{$j = k$ to $\min(k+\ell+m, n)$}
        \State $B_n[k, j] \gets \tilde{\mathbf{d}}_{k}[j-k+1]$ (Store upper banded part)
    \EndFor

    \State \textbf{Step 4: Update matrices}
    \State $[J, K, E_s, X_s, Y_s, Z_s] \gets \text{UpdateMatrices}(J, K, E_s, X_s, Y_s, Z_s, \bar{\mathbf{k}}_{k}, \mathbf{b}_{k}, \tau_k)$ (Algorithm \ref{algo:update_hmbpsm})
\EndFor

\State \textbf{Final step:}
\State $B_n[n, n] \gets A_{n-1}[n, n]$
\State \Return $B_n + \tril(UV_n^\top,-1) + \triu(W_n \tilde S^\top,1)$ and  $\boldsymbol{\tau}$.
\end{algorithmic}
\end{algorithm}

\begin{algorithm}
\caption{FormHouseholderVector}
\label{algo:form_householder}
\begin{algorithmic}
\State \textbf{Input} $A_{k-1} \in \mathbb{R}^{(n-k+1) \times (n-k+1)}$: the current trailing submatrix at step $k$, expressible as the remaining block of the original matrix plus accumulated structured perturbations, i.e., $A_{k-1} = A[k:n, k:n] + P_{k-1}$ where $P_{k-1} \in \mathcal{P}(A[k:n, k:n])$;
$U_k \in \mathbb{R}^{(n-k+1) \times r}$: $U[k:n, :]$;$V_k \in \mathbb{R}^{(n-k+1) \times r}$: $V[k:n, :]$;
$\ell$: lower bandwidth.
\State \textbf{Output} $\bar{\mathbf{k}} \in \mathbb{R}^r$, $\mathbf{b} \in \mathbb{R}^{n-k+1}$, $\tau\in \mathbb{R}$ s.t. $\bfy$: = $ \bfe_1 + \tilde U_k\bar{\mathbf{k}} + \mathbf{b}$ is the Householder vector where $\tilde U_k \in \mathbb{R}^{(n-k+1) \times r}$ with $\tilde U_k[1,:]=\mathbf{0}$ and $\tilde U_k[2:n-k+1,:]=U[k+1:n,:]$. $I-\tau \bfy \bfy^\top$ is the Householder transformation. (Note that the expression for $\bfy$ here differs from that in Eq.~(\ref{eq:y_k}); both forms are valid, and it is straightforward to convert between them.)
\State $\tilde \bfb_k \gets $ Eq. (\ref{eq:tilde_b_k}), \text{computed using the entries available within the input $A_{k-1}$.}
\State $\alpha_k \gets $ Eq. (\ref{eq:alpha_k}), computed entirely from $A_{k-1}$.
\State $\mathbf{b} \gets $ Eq. (\ref{eq:hv_update_1})
\State $\bar{\mathbf{k}} \gets $ Eq. (\ref{eq:hv_update_2}) 
\State $\mathbf{b}$ is only nonzero in entries $2$ through $\min(\ell+1, n-k+1)$.
\State $\tau \gets 2/\bfy^\top\bfy$
\State \Return $\bar{\mathbf{k}}$, $\mathbf{b}$, $\tau$
\end{algorithmic}
\end{algorithm}

\begin{algorithm}
\caption{ComputeRowUpdate}
\label{algo:compute_row}
\begin{algorithmic}
\State \textbf{Input} Current submatrix $A_{k-1}$,
  matrices $U, \tilde S$, a
$\mathbf{k}_{k}$, and scaling vector $\mathbf{b}_{k}$, $\tau_k$.
\State \textbf{Output} semiseparable part $\tilde \bfw \in \mathbb{R}^{r+p}$ and banded part $\tilde{\mathbf{d}} \in \mathbb{R}^{n-k+1}$ 

\State $\tilde \bfw^\top \gets $ Eq.(\ref{eq:row_updtae_1})

\State $\tilde{\mathbf{d}}^\top \gets $ Eq.(\ref{eq:row_updtae_2})
\State $\tilde{\mathbf{d}}$ is only nonzero in entries $1$ through $\min(\ell+m+1, n-k+1)$.
\State \Return $\tilde \bfw, \tilde{\mathbf{d}}$
\end{algorithmic}
\end{algorithm}

\begin{algorithm}
\caption{UpdateMatrices}
\label{algo:update_hmbpsm}
\begin{algorithmic}
\State \textbf{Input} $J, K, E_s, X_s, Y_s, Z_s$: current parameter matrices;
$\bar{\mathbf{k}}_{k}$, $\mathbf{b}_{k}$, $\tau_k$.
\Ensure Updated $J, K, E_s, X_s, Y_s, Z_s$

\State Compute updates using formulas from Lemma~\ref{Lemma:Structure_preserve} proof:
\State $J_{\text{new}} \gets $Eq. (\ref{eq:submatrix_update_J})
\State $K_{\text{new}} \gets $Eq. (\ref{eq:submatrix_update_K})
\State $E_s^{\text{new}} \gets $Eq. (\ref{eq:submatrix_update_E})
\State $X_s^{\text{new}} \gets $Eq. (\ref{eq:submatrix_update_X})
\State $Y_s^{\text{new}} \gets $Eq. (\ref{eq:submatrix_update_Y})
\State $Z_s^{\text{new}} \gets $Eq. (\ref{eq:submatrix_update_Z})

\State \Return $J_{\text{new}}, K_{\text{new}}, E_s^{\text{new}}, X_s^{\text{new}}, Y_s^{\text{new}}, Z_s^{\text{new}}$
\end{algorithmic}
\end{algorithm}

\subsubsection*{Complexity Analysis}

The algorithm runs for $n-1$ steps. The cost per step can be expressed as a polynomial in terms of $r$, $p$, $\ell$, and $m$. Since these are constants independent of $n$, the total complexity is $O(n)$ as $n \rightarrow \infty$. The memory footprint is also $O(n)$, as we store only the generators and banded components.

\begin{remark}\label{Remark:lookup}

To maintain the $O(1)$ per-step complexity in Algorithm \ref{algo:fastqr}, two key quantities must be computed efficiently during the Householder updates:

    Inner product matrix $U[k:n,:]^\top U[k:n,:]$: The computation of many intermediate vectors requires evaluating expressions like $U[k:n,:]^\top \bfy_k $, which involves $U[k:n,:]^\top U[k:n,:]$. So we precompute a lookup table:
    \[
    \text{UU\_lookup}[k] = U[k:n,:]^\top U[k:n,:] \quad \text{for } k = 1,\dots,n
    \]
    This can be computed in $O(nr^2)$ time via a backward accumulation.

    Partial sum $\sum_{t=1}^{j} U[t,:] W[t,:]^\top$: The update of the upper triangular part for factor matrix requires this sum(Eq.\ref{eq:UtA_relation}). We precompute:
    \[
    \text{UV\_lookup}[j] = \sum_{t=1}^{j} U[t,:] W[t,:]^\top \quad \text{for } j = 1,\dots,n-1
    \]
    This is computed in $O(nrp)$ time via forward accumulation.

Both precomputations require $O(n)$ total time and enable $O(1)$ access to the required quantities at each step of the factorization, thus preserving the overall $O(n)$ complexity.
\end{remark}

\begin{remark}
    In certain special cases the look-up tables can be pre-computed analytically. for example, when $U = [1, 1/2, 1/3, \dots, 1/n]^\top$ we have
\[
UU\_lookup[k] = \sum_{j=k}^n {1 \over j^2} = \psi'(k)-\psi'(n+1)
\]
for the polygamma function \cite[5.15]{OlverNIST}
\[
\psi'(k) = \sum_{j=0}^\infty {1\over (k+j)^2}
\]
which can be computed effectively via asymptotics and recurrence relationships \cite[5.12]{OlverNIST}.
\end{remark}

\subsection{Fast Solver for BPS Matrices}
\label{sec:solver}

Theorem \ref{thm:structure_preserve} not only enables an efficient QR factorization but also facilitates a complete direct solver for linear systems of the form \( A\mathbf{x} = \mathbf{b} \), where \( A \) is a BPS matrix. The solver consists of two phases after the QR factorization \( A = QR \):
1. Application of \( Q^\top \) to the right-hand side vector \( \mathbf{b} \) to form \( \mathbf{c} = Q^\top \mathbf{b} \).
2. Solution of the upper triangular system \( R\mathbf{x} = \mathbf{c} \) via backward substitution.

We now present \( O(n) \) algorithms for both phases, leveraging the structured representation of the factorization output by Algorithm \ref{algo:fastqr}.

\subsubsection{Fast Application of \( Q^\top \)}

The orthogonal matrix \( Q \) is represented as a product of Householder transformations:
\[
Q = (I - \tau_1 \bfy_1 \bfy_1^\top)(I - \tau_2 \bfy_2 \bfy_2^\top) \cdots (I - \tau_{n-1} \bfy_{n-1} \bfy_{n-1}^\top).
\]
Applying \( Q^\top \) to a vector \( \mathbf{b} \) thus requires computing:
\[
Q^\top\mathbf{b} = (I - \tau_{n-1} \bfy_{n-1} \bfy_{n-1}^\top) \cdots (I - \tau_2 \bfy_2 \bfy_2^\top)(I - \tau_1 \bfy_1 \bfy_1^\top)\mathbf{b}.
\]

The Householder vectors \( \bfy_k \) are stored in the factor matrix \( F \) according to the normalization convention established in Section~\ref{sec:factors}:
\[
\bfy_k[j] =
\begin{cases}
0, & j < k \\
1, & j = k \\
F[j,k], & j > k
\end{cases}
\quad \text{for } k = 1, \ldots, n-1.
\]

From Theorem \ref{thm:structure_preserve}, the factor matrix \( F \) admits the BPS representation:
\begin{equation*}
F = B_n + \tril(U V_n^\top, -1) + \triu(W_n \tilde S^\top, 1),
\end{equation*}
where \( U, V_n \in \mathbb{R}^{n \times r} \), \( W_n, \tilde S \in \mathbb{R}^{n \times (r+p)} \), and \( B_n \) is banded with lower bandwidth \( \ell \) and upper bandwidth \( \ell+m \).

This structure implies that each Householder vector \( \bfy_k \) can be expressed as:
\begin{equation}
\bfy_k = \mathbf{e}_1 + \tilde U_k V_n[k,:] + \mathbf{d}_k,
\label{eq:y_structure}
\end{equation}
where:

\( \tilde U_k \in \mathbb{R}^{n \times r} \) satisfies \( \tilde U_k[1:k,:] = \mathbf{0} \) and \( \tilde U_k[k+1:n,:] = U[k+1:n,:] \),

and \( \mathbf{d}_k \in \mathbb{R}^n \) is non-zero only in entries \( k+1 \) through \( \min(k+\ell, n) \), with \( \mathbf{d}_k[j] = B_n[j,k] \) for \( j = k+1, \ldots, \min(k+\ell, n) \).

Algorithm~\ref{algo_applyQ} exploits this structure to compute \( Q^\top\mathbf{b} \) in \( O(n) \) operations by maintaining a compressed representation of the intermediate vectors throughout the transformation process.

\begin{algorithm}
\caption{Fast Application of \( Q^\top \) to a Vector}\label{algo_applyQ}
\begin{algorithmic}
\State \textbf{Input:} Factor matrix $F \in {\mathrm{BPS}}_{(r,r+p),(\ell,\ell+m)}^{n \times n}$, scaling vector \( \boldsymbol{\tau} = [\tau_1, \ldots, \tau_{n-1},0]^\top \in \mathbb{R}^{n} \), right-hand side \( \mathbf{b} \in \mathbb{R}^n \).
\State \textbf{Output:} \( \mathbf{c} = Q^\top\mathbf{b} \in \mathbb{R}^n \)

\State Initialize:

\( O \gets \mathbf{0}_{n \times r} \): Storage for accumulated low-rank updates

\( G \gets \mathbf{0}_{n \times (\ell+1)} \): Storage for banded component updates

\( \mathbf{h} \gets \mathbf{0}_r \): Accumulator for semiseparable component

Let \( \mathbf{o}_i \) denote the \( i \)-th column of \( O \)

Let \( \mathbf{g}_i \) denote the \( i \)-th column of \( G \)

\State Express initial vector: \( \mathbf{b}^{(0)} = \mathbf{b} + \tilde U_0 \mathbf{h} + \sum_{j=1}^r \mathbf{o}_j + \sum_{j=1}^{l+1} \mathbf{g}_j \)

\For{\( k = 1 \) to \( n-1 \)}
    \State Compute inner product: \( c \gets \bfy_k^\top \mathbf{b}^{(k-1)} \) (exploit BPS structure of \( \bfy_k \) and precompute some lookup tables for for \( O(1) \) computation)

    \State Update low-rank storage: \( O[k,:] \gets U[k,:] \odot \mathbf{h} \) (element-wise multiplication)

    \State Update semiseparable accumulator: \( \mathbf{h} \gets \mathbf{h} - \tau_k c V_n[k,:] \)

    \State Update banded component:
    \State \( G[k,1] \gets -\tau_k c \) (diagonal contribution)
    \For{\( t = 1 \) to \( \min(\ell, n-k) \)}
        \State \( G[k+t,t+1] \gets -\tau_k c \cdot B_n[k+t,k] \) (subdiagonal contributions)
    \EndFor

    \State Current representation: \( \mathbf{b}^{(k)} = \mathbf{b} + \tilde U_k \mathbf{h} + \sum_{j=1}^r \mathbf{o}_j + \sum_{j=1}^{\ell+1} \mathbf{g}_j \)
\EndFor

\State Compute final result explicitly: \( \mathbf{c} \gets \mathbf{b} + \tilde U_{n-1} \mathbf{h} + \sum_{j=1}^r \mathbf{o}_j + \sum_{j=1}^{l+1} \mathbf{g}_j \)

\State \Return \( \mathbf{c} \)
\end{algorithmic}
\end{algorithm}

\begin{mytheorem}
Algorithm~\ref{algo_applyQ}  computes \( \mathbf{c} = Q^\top\mathbf{b} \) in \( O(n) \) operations.
\end{mytheorem}

\begin{proof}
The proof proceeds by induction on the transformation steps. Let \( \mathbf{b}^{(0)} = \mathbf{b} \) and assume that after \( k-1 \) steps, the algorithm maintains the representation:
\[
\mathbf{b}^{(k-1)} = \mathbf{b} + \tilde U_{k-1}\mathbf{h}^{(k-1)} + \sum_{j=1}^r \mathbf{o}_j^{(k-1)} + \sum_{j=1}^{\ell+1} \mathbf{g}_j^{(k-1)},
\]
where the superscripts on $\mathbf{h}$, $\mathbf{o}$, and $\mathbf{g}$ denote the state after the \( (k-1) \)-th iteration.

The \( k \)-th Householder transformation gives:
\[
\mathbf{b}^{(k)} = (I - \tau_k \bfy_k \bfy_k^\top) \mathbf{b}^{(k-1)} = \mathbf{b}^{(k-1)} - \tau_k (\bfy_k^\top \mathbf{b}^{(k-1)}) \bfy_k.
\]

Substituting the structured form of \( \bfy_k \) from \eqref{eq:y_structure} and the inductive representation:
\begin{align*}
\mathbf{b}^{(k)} &= \mathbf{b} + \tilde U_{k-1} \mathbf{h}^{(k-1)} + \sum_{j=1}^r \mathbf{o}_j^{(k-1)} + \sum_{j=1}^{\ell+1} \mathbf{g}_j^{(k-1)} \\
&\quad - \tau_k c (\bfe_1 + \tilde U_{k-1} \bar{\mathbf{v}}_k + \mathbf{d}_k) \\
&= \mathbf{b} + \tilde U_{k} (\mathbf{h}^{(k-1)} - \tau_k c \bar{\mathbf{v}}_k) \\
&\quad + \left( \sum_{j=1}^r \mathbf{o}_j^{(k-1)} + (\tilde U_{k-1} - \tilde U_k) \mathbf{h}^{(k-1)} \right) \\
&\quad + \left( \mathbf{g}_1^{(k-1)} - \tau_k c \bfe_1 \right) + \left( \sum_{j=2}^{\ell+1} \mathbf{g}_j^{(k-1)} - \tau_k c \mathbf{d}_k \right).
\end{align*}

The algorithm updates precisely these components:

\( \mathbf{h}^{(k)} = \mathbf{h}^{(k-1)} - \tau_k c \bar{\mathbf{v}}_k \),

\( O[k,:] = U_n[k,:] \odot \mathbf{h}^{(k-1)} \) captures \( (\tilde U_{k-1} - \tilde U_k) \mathbf{h}^{(k-1)} \),

Banded updates in \( G \) capture the remaining terms.

Thus, the representation is maintained correctly throughout all \( n-1 \) steps. Each step requires \( O(1) \) operations due to the constant-bounded parameters \( r, p, \ell, m \), yielding overall \( O(n) \) complexity.
\end{proof}

\subsubsection{Fast Backward Substitution}

After computing \( \mathbf{c} = Q^\top \mathbf{b} \), we solve the upper triangular system \( R\mathbf{x} = \mathbf{c} \), where \( R = \triu(F) \) inherits the BPS structure of \( F \). Specifically, the upper triangular part of \( F \) satisfies:
\[
R = B_R + \triu(W_n \tilde S^\top, 1),
\]
where \( B_R = \triu(B_n) \) is the upper triangular part of the banded component of the factor matrix, maintaining upper bandwidth \( \ell+m \).

Algorithm~\ref{algo_backsub}, which is equivalent to the one introduced in~\cite{olver2013fast}, exploits this structure to perform backward substitution in \( O(n) \) operations by maintaining a running sum for the semiseparable contributions.

\begin{algorithm}
\caption{Fast Backward Substitution for Structured \( R \)}\label{algo_backsub}
\begin{algorithmic}
\State \textbf{Input:} Upper triangular matrix \( R = \triu(F) \) in structured form; transformed right-hand side \( \mathbf{c} \in \mathbb{R}^n \)
\State \textbf{Output:} Solution \( \mathbf{x} \in \mathbb{R}^n \) satisfying \( R\mathbf{x} = \mathbf{c} \)

\State Initialize:

\( \mathbf{x} \gets \mathbf{0}_n \): solution vector

\( \mathbf{s} \gets \mathbf{0}_{r+p} \): Accumulator for semiseparable contributions

\For{\( j = n \) down to \( 1 \)}
    \State Initialize  accumulated sum: \( \text{sum} \gets 0 \)

    \State Add semiseparable contribution: \( \text{sum} \gets \text{sum} + W_n[j,:]^\top \mathbf{s} \)

    \State Add banded contributions:
    \For{\( k = j+1 \) to \( \min(j+\ell+m, n) \)}
        \State \( \text{sum} \gets \text{sum} + B_n[j,k]  \mathbf{x}[k] \)
    \EndFor

    \State Solve for \( x_j \): \( \mathbf{x}[j] \gets (\mathbf{c}[j] - \text{sum}) / B_n[j,j] \)

    \State Update semiseparable accumulator: \( \mathbf{s} \gets \mathbf{s} + \tilde S[j,:]^\top \mathbf{x}[j] \)
\EndFor

\State \Return \( \mathbf{x} \)
\end{algorithmic}
\end{algorithm}

\begin{mytheorem}
Algorithm~\ref{algo_backsub}  solves \( R\mathbf{x} = \mathbf{c} \) in \( O(n) \) operations.
\end{mytheorem}

\begin{proof}
For completeness we include the proof from~\cite{olver2013fast}. The algorithm implements standard backward substitution while exploiting the structure of \( R \). For each index \( j \) from \( n \) down to \( 1 \), the equation:
\[
R[j,j] x_j + \sum_{k=j+1}^n R[j,k] x_k = c_j
\]
is solved for \( x_j \).

The key insight is that the off-diagonal entries \( R[j,k] \) for \( k > j \) can be decomposed as:
\[
R[j,k] = B_n[j,k] + W_n[j,:]^\top \tilde S[k,:].
\]

The banded contributions \( B_R[j,k] \) are non-zero only for \( k = j+1, \ldots, \min(j+l+m, n) \), requiring \( O(1) \) operations per row. The semiseparable contributions are accumulated in the vector \( \mathbf{s} \), which stores:
\[
\mathbf{s} = \sum_{k=j+1}^n x_k \tilde S[k,:].
\]

At step \( j \), the product \( W_n[j,:] ^\top \mathbf{s} \) thus captures all semiseparable contributions from previously computed solution components. After computing \( x_j \), the accumulator is updated to include its contribution.

Each iteration requires \( O(1) \) operations, yielding overall \( O(n) \) complexity. The correctness follows by induction from \( j = n \) down to \( 1 \).
\end{proof}

\subsubsection{Overall Solver Complexity}

Combining the QR factorization (Algorithm \ref{algo:fastqr}), the fast application of \( Q^\top \) (Algorithm~\ref{algo_applyQ}), and the fast backward substitution (Algorithm~\ref{algo_backsub}) yields a complete direct solver for BPS linear systems with \( O(n) \) complexity.

\begin{corollary}
For a BPS matrix \( A \in \mathbb{R}^{n \times n} \) with constant-bounded ranks and bandwidths, the linear system \( A\mathbf{x} = \mathbf{b} \) can be solved in \( O(n) \) operations using the QR-based approach.
\end{corollary}

\section{Numerical results}
\label{sec:experiments}

To validate the theoretical complexity and demonstrate the practical efficiency of our proposed algorithms, we implemented the fast QR factorization, the complete linear solver, and the fast RQ product computation for symmetric BPS matrices in Julia. The implementation is publicly available in the {SemiseparableMatrices.jl} package~\citep{SemiseparableMatrices2024}, providing an open-source resource for the scientific computing community. Computations were carried out on a MacBook Air equipped with an Apple M2 chip (8-core CPU, 8 GB RAM), without GPU acceleration or access to external computing resources.

\subsection{Linear Complexity Verification}

To verify that the theoretical $O(n)$ complexity holds for matrices with different structural characteristics, we test four parameter sets $(\ell,m,r,p)$ representing different balances between the banded and semiseparable components. Minimal case: $(1,1,1,1)$; balanced case: $(4,5,5,4)$; band-dominated case: $(8,8,1,1)$; semiseparable-dominated case: $ (1,1,8,8)$.

\begin{figure}[htbp]
  \centering
  \begin{minipage}{0.48\textwidth}
    \centering
    \includegraphics[width=\textwidth]{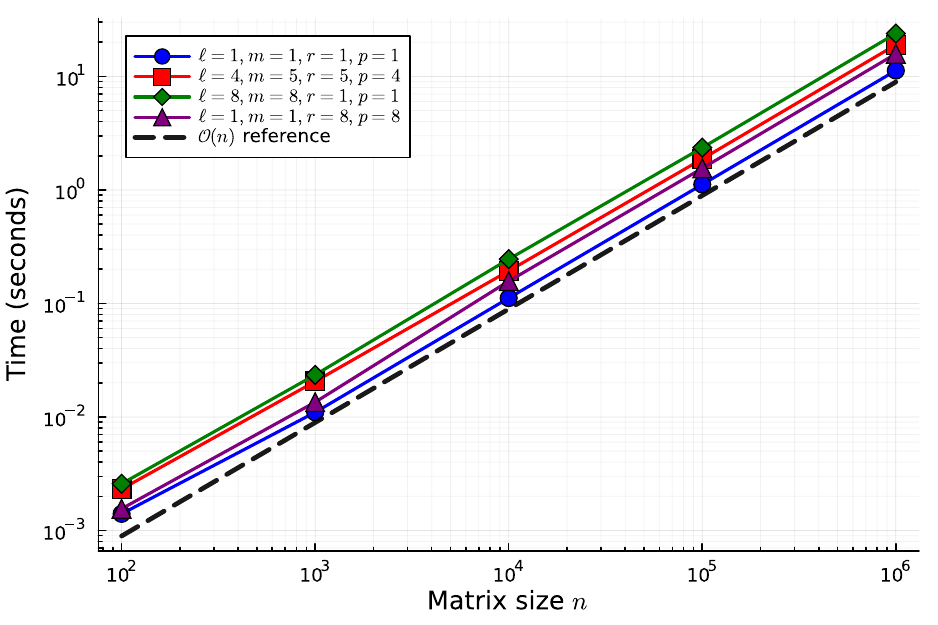}
    \vspace{0.1cm}
    \\ \textbf{(a)} QR factorization phase.
    \label{fig:qr_fact_only}
  \end{minipage}
  \hfill
  \begin{minipage}{0.48\textwidth}
    \centering
    \includegraphics[width=\textwidth]{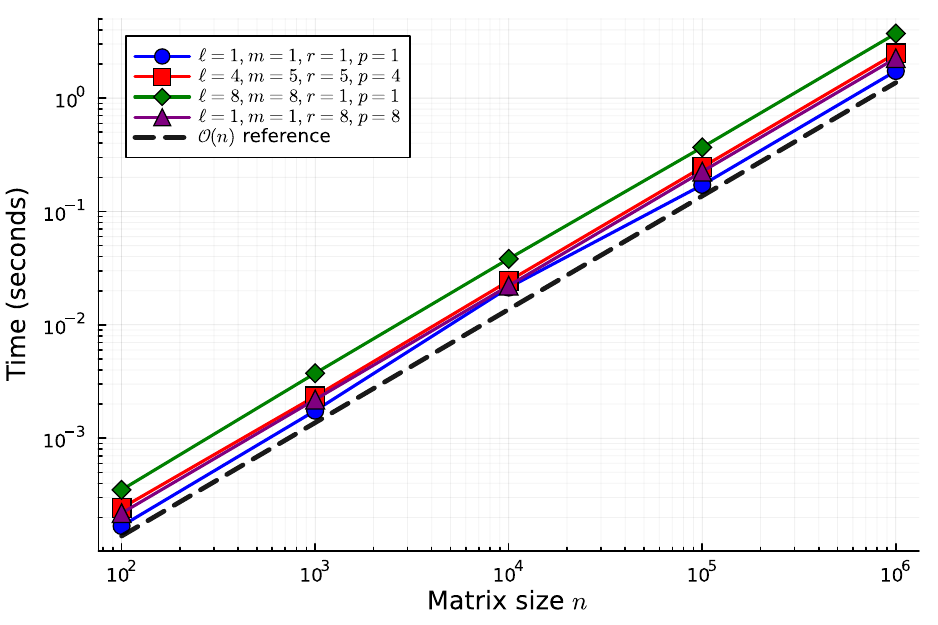}
    \vspace{0.1cm}
    \\ \textbf{(b)} solution phase.
    \label{fig:qr_solve_only}
  \end{minipage}
  \vspace{0.3cm}
  \caption{Log-log plots of execution times versus matrix size $n$ for four parameter sets $(\ell, m, r, p)$, demonstrating optimal linear complexity. (a) Execution time for the structured QR factorization; (b) Execution time for the solve phase (applying $Q^\top$ and backward substitution). The dashed lines have a slope of 1, indicating perfect linear scaling $O(n)$.}
  \label{fig:qr_scaling}
\end{figure}

Figure~\ref{fig:qr_scaling} demonstrates the linear time complexity of our complete QR solver for BPS linear systems. The execution time scales as $O(n)$ across five orders of magnitude, from $n=100$ to $n=10^6$, for four different parameter configurations. All curves align with the reference line of slope 1, confirming the analysis in Section~\ref{sec:algorithms}.

While the computational time scales as $O(n)$, the underlying constant factor depends polynomially on the structural parameters $r, p, \ell,$ and $m$. Specifically, the cost per step is dominated by updating the perturbation space $\mathcal{P}(A)$, leading to an overall complexity of $O(n \cdot (r^2 + rp + r(\ell+m) + \ell p + \ell(\ell+m)))$ for the complete solver. This strict parameter dependence accounts for the vertical offsets among the curves in Figure~\ref{fig:qr_scaling} and Figure~\ref{fig:rq_scaling}: configurations with larger semiseparable ranks or wider bandwidths exhibit an elevated constant overhead, while maintaining the identical linear scaling slope of 1.

\subsection{Numerical Stability of the QR Solver}
\label{subsec:stability}

\begin{figure}[htbp]
\centering
\includegraphics[width=0.8\textwidth]{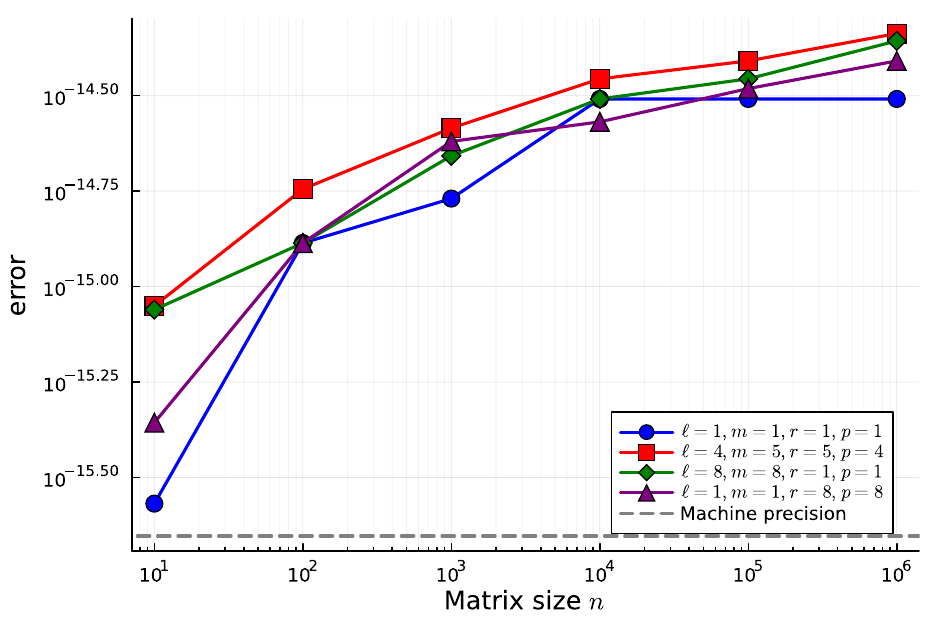}
\caption{Log-log plot  of the maximum residual $\max|A\mathbf{x} - \mathbf{b}|$ versus matrix size $n$ for four parameter sets $(\ell, m, r, p)$. Here $\mathbf{x}$ is computed by our complete QR solver (QR factorization followed by application of $Q^\top$ and backward substitution). The dashed gray line indicates the machine precision $\epsilon_{\mathrm{machine}} \approx 2.22 \times 10^{-16}$.}
\label{fig:stability}
\end{figure}

The numerical stability of a QR-based linear solver is crucial for its practical utility. Figure~\ref{fig:stability} demonstrates the excellent stability properties of our structured QR solver across a wide range of matrix sizes and structural parameters. We measure the accuracy through the maximum norm of the residual vector:
\[
\max|A\mathbf{x} - \mathbf{b}|,
\]
where $\mathbf{x}$ is the solution computed by our solver for a randomly generated right-hand side $\mathbf{b}$.

To isolate the numerical behavior of the QR factorization itself, all test matrices are constructed to be well-conditioned; more precisely, they are diagonally dominant. This ensures that the observed residuals are not influenced by ill-conditioning of the linear system, but instead reflect only the numerical errors introduced by the factorization and solution process.

For all parameter configurations and across six orders of magnitude in $n$ (from $n=10$ to $n=10^6$), the computed residuals remain very small. The growth of the error is gradual: for smaller matrices, it is on the order of machine precision ($ \epsilon_{\mathrm{machine}} \approx 2.22 \times 10^{-16}$), while for the largest matrices, the maximum error corresponds to approximately 14 digits of accuracy. Even for the most challenging case (semiseparable-dominated with $r=p=8$) and the largest $n=10^6$, the algorithm preserves a maximum residual error of less than $5 \times 10^{-15}$, demonstrating its robustness and reliable numerical stability across a wide range of matrix structures and sizes.

\begin{remark}
Measuring the accuracy of large-scale QR factorization presents a methodological challenge. Traditional approaches like comparing the factor matrix $\tilde{F}$ generated by our structured QR algorithm against the factor matrix $F_{\text{dense}}$ produced by a standard dense QR routine, or evaluating the reconstruction error $\|\tilde{Q}\tilde{R} - A\|$ where $\tilde{Q},\tilde{R}$ are extracted from $\tilde{F}$ would require $O(n^3)$ operations to compute the dense equivalent matrices and therefore becomes infeasible for $n > 10^4$.

We overcome this limitation by leveraging the structure of BPS matrices: after obtaining the solution $\mathbf{x}$ via our $O(n)$ solver, we compute the residual $A\mathbf{x} - \mathbf{b}$ using the $O(n)$ matrix-vector multiplication algorithm for BPS matrices~\citep{chandrasekaran2002fast}. This approach allows us to verify accuracy for matrices as large as $n=10^6$ while maintaining overall $O(n)$ computational cost.
\end{remark}

\subsection{Comparison with HODLR QR}

\begin{figure}[htbp]
  \centering
  \begin{minipage}{0.48\textwidth}
    \centering
    \includegraphics[width=\textwidth]{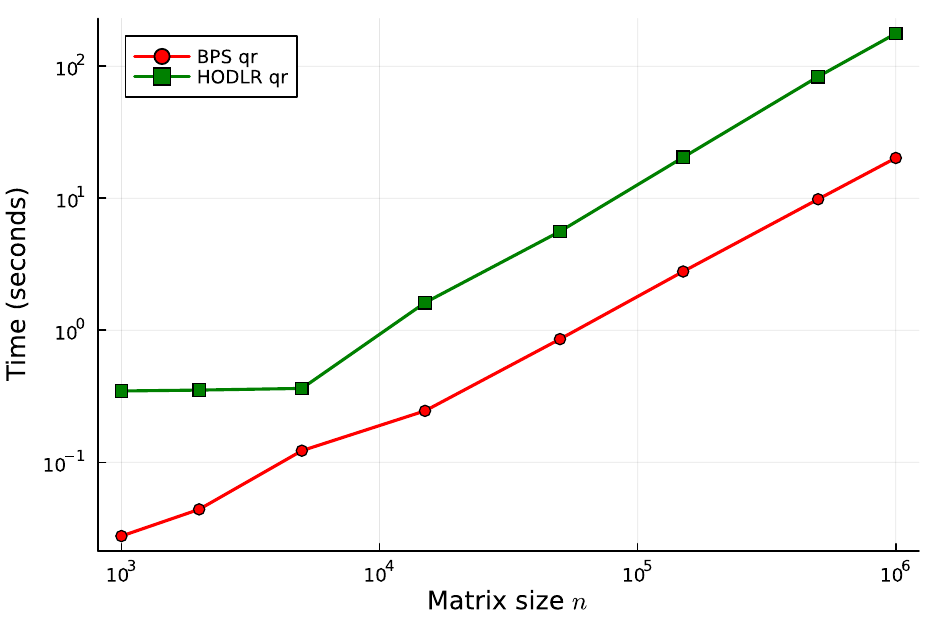}
    \vspace{0.1cm}
    \\ \textbf{(a)} QR factorization phase.
    \label{fig:qr_fact_comp}
  \end{minipage}
  \hfill
  \begin{minipage}{0.48\textwidth}
    \centering
    \includegraphics[width=\textwidth]{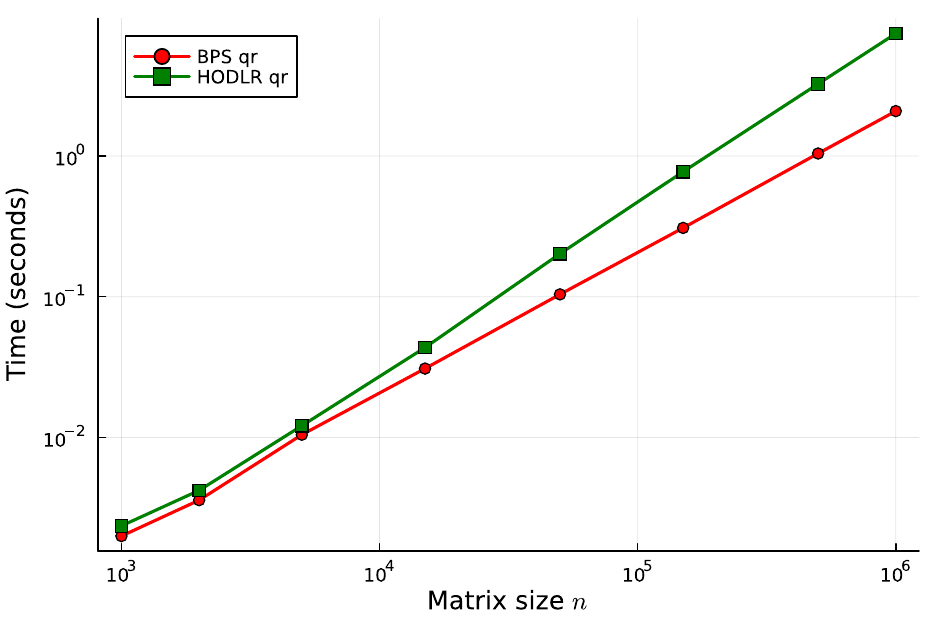}
    \vspace{0.1cm}
    \\ \textbf{(b)} solution phase.
    \label{fig:qr_solve_comp}
  \end{minipage}
  \vspace{0.3cm}
  \caption{Log-log performance comparisons between our proposed fast BPS QR algorithm and the HODLR QR implementation from~\citep{massei2020hm} (with parameters $\ell=4$, $m=5$, $r=2$, $p=3$). Panel (a) illustrates the execution times for the structured QR factorization phase; Panel (b) presents the execution times for the solve phase (applying $Q^\top$ and backward substitution).}
  \label{fig:comparison}
\end{figure}

We compare our fast QR factorization against the state-of-the-art HODLR (Hierarchically Off-Diagonal Low-Rank) QR implementation from the hm-toolbox~\citep{massei2020hm}. Hm-toolbox provides efficient MATLAB routines for various structured matrices, including HODLR and HSS matrices, and represents one of the most mature implementations for hierarchical matrix computations. 

Crucially, while general HODLR-based QR factorization algorithms typically incur an asymptotic computational complexity of $O(n \log n)$ due to their hierarchical partitioning schemes, our proposed specialized algorithm achieves an optimal linear complexity of $O(n)$ by directly preserving and tracking the BPS structure throughout the reflections.

Figure~\ref{fig:comparison} compares the execution times of QR factorization for BPS matrices obtained using the two approaches. Our algorithm consistently exhibits superior performance as the matrix size increases. This deterministic scaling advantage arises from its explicit exploitation of the underlying BPS structure, which altogether avoids the $O(n \log n)$ hierarchical overhead associated with general low-rank approximations.

Moreover, the performance gap widens with increasing $n$, confirming that our method is particularly well suited for large-scale problems. For $n = 1{,}000{,}000$, our implementation achieves an approximate $9\times$ speedup over the HODLR approach, highlighting the practical benefits of the proposed specialized algorithm.

\subsection{LAPACK Compatibility and Hybrid Solver Performance} \label{sec:lapack_comp_exps}
A key practical advantage of our Householder-based formulation is its native compatibility with standard LAPACK computational routines. To evaluate performance across different hardware utilization regimes, we benchmark three distinct algorithms for solving $Ax = b$ with dimension $n$ ranging from $500$ to $10,000$ under fixed parameters $\ell=4, m=5, r=2,$ and $p=3$. 
Specifically, we compare our structured $O(n)$ BPS algorithm against two canonical dense LAPACK implementations: 
\begin{enumerate}[(i)]
    \item the Dense Unblocked solver (level-2 BLAS), which processes Householder reflections sequentially; and 
    \item the Dense Compact WY solver (level-3 BLAS), which accumulates reflections into block representations ($Q = I + V T V^\top$) to maximize cache reuse and parallel processing.
\end{enumerate}
The execution times are separately measured for the QR factorization phase and the solve phase (applying $Q^\top$ and backward substitution), as presented in Figure~\ref{fig:lapack_compare}.

\begin{figure}[htbp]
  \centering
  \begin{minipage}{0.48\textwidth}
    \centering
    \includegraphics[width=\textwidth]{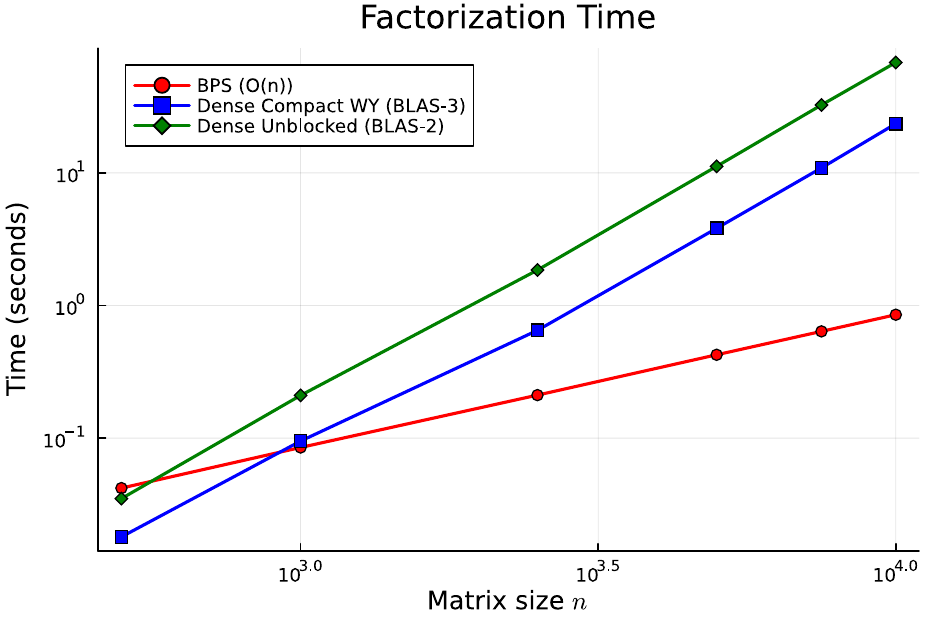}
    \vspace{0.1cm}
    \\ \textbf{(a)} QR factorization phase.
    \label{fig:lapack_qr_comp}
  \end{minipage}
  \hfill
  \begin{minipage}{0.48\textwidth}
    \centering
    \includegraphics[width=\textwidth]{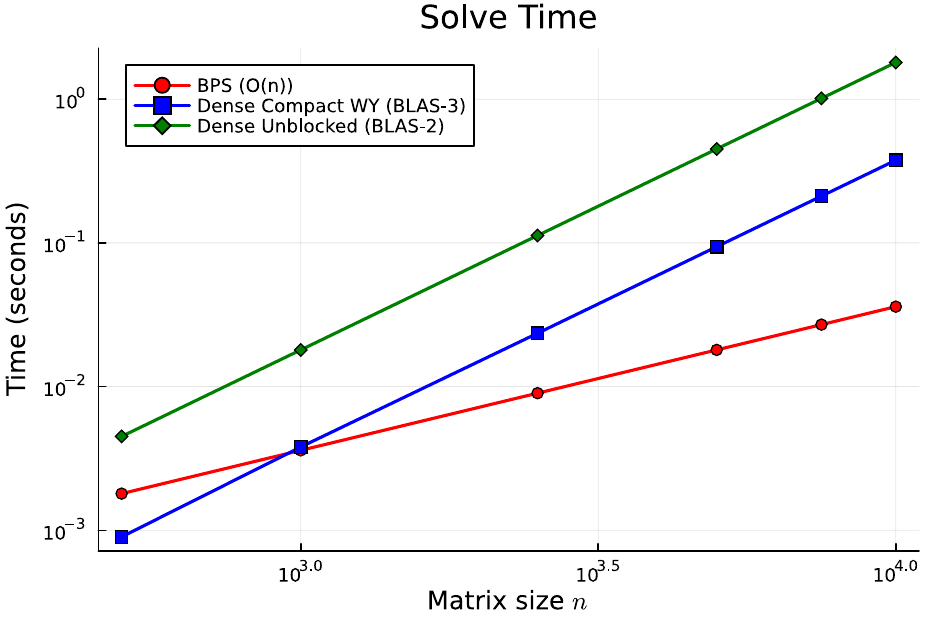}
    \vspace{0.1cm}
    \\ \textbf{(b)} Solve phase ($Q^\top b$ and backward substitution).
    \label{fig:lapack_solve_comp}
  \end{minipage}
  \vspace{0.3cm}
  \caption{Performance scaling comparison between our proposed structured $O(n)$ BPS solver and LAPACK dense solvers (both Unblocked and Compact WY BLAS-3 implementations). Panel (a) shows the execution times for the QR factorization phase; Panel (b) illustrates the execution times for the solve phase.}
  \label{fig:lapack_compare}
\end{figure}

The benchmarks highlight distinct asymptotic scaling behavior and computational efficiency between the factorization and solve phases. 

During the QR factorization phase, both dense LAPACK solvers suffer from $O(n^3)$ computational complexity. Although the Compact WY (BLAS-3) variant significantly outperforms the Unblocked (BLAS-2) variant by leveraging cache-friendly matrix-matrix multiplication, both dense approaches scale poorly as $n$ increases, requiring $23.50$ seconds and $68.00$ seconds at $n = 10,000$, respectively. In contrast, our proposed BPS factorization scales linearly as $O(n)$, executing in just $0.85$ seconds at $n = 10,000$---achieving a speedup of over $27\times$ against Compact WY and $80\times$ against the Unblocked solver.

A similar advantage appears in the solve phase. For smaller matrix sizes ($n \le 500$), the Compact WY solver achieves slightly lower latency due to highly optimized level-3 BLAS primitives operating on contiguous memory. However, as $n$ grows past $1,000$, the $O(n)$ complexity of our BPS solve phase demonstrates its asymptotic superiority over the dense $O(n^2)$ back-solvers. At $n = 10,\!000$, our BPS solver completes the solve phase in only $0.036$ seconds, outperforming the Dense Compact WY solver ($0.376$ seconds) and the Dense Unblocked solver ($1.800$ seconds) by a wide margin. 

Overall, these results confirm that our structured $O(n)$ factorization delivers theoretical linear scaling while achieving superior practical speed, outperforming highly optimized, level-3 BLAS dense routines. 

\begin{remark}
A conclusion of these timings is that on the CPU it is in practice more efficient to use the $O(n)$ BPS QR solve than to switch to a more optimized $O(n^2)$ LAPACK QR solve. However, it is also possible to compute the $O(n)$ BPS QR factorization on a CPU with GPU-based $O(n^2)$ LAPACK QR solves, which may have substantial performance benefits for moderate $n$.
\end{remark}

\section{Conclusions}
\label{sec:conclusions}

In this paper, we have established a fundamental theoretical result for BPS matrices and developed efficient algorithms based on this foundation. Our main contribution is the explicit proof that the standard QR factorization of a BPS matrix via Householder reflections preserves the BPS structure, with precisely characterized ranks and bandwidths in the resulting factor matrix. This theoretical insight enabled the design of a complete, Householder-based $O(n)$ direct solver for BPS linear systems that is natively compatible with LAPACK storage layouts, comprising: (1) a structure-preserving QR factorization algorithm (Algorithm~\ref{algo:fastqr}); (2) an efficient $O(n)$ application of $Q^\top$ (Algorithm~\ref{algo_applyQ}); and (3) a fast backward substitution routine (Algorithm~\ref{algo_backsub}).

Furthermore, for symmetric BPS matrices, we have shown that the reverse $RQ$ product also maintains the BPS structure. This establishes the algebraic closure of the BPS matrix class under symmetric orthogonal similarity transformations and provides a crucial computational building block for more complex matrix algebra. This theoretical guarantee led to the development of another $O(n)$ algorithm (Algorithm~\ref{algo:fastrq}) for computing the $RQ$ product without explicitly forming the dense orthogonal matrix $Q$.

The numerical experiments confirm the linear scaling and accuracy of our approach, while demonstrating significant performance advantages over existing HODLR-based methods. Additionally, our solver's native LAPACK compatibility allows for a highly efficient hybrid approach that bridges the gap between structured and dense solvers. Our implementation in the SemiseparableMatrices.jl package provides the scientific computing community with efficient, open-source tools for working with this important class of structured matrices.

\subsection*{Future Work}

A compelling extension involves applying our methodology to specific blocked banded matrices arising in $hp$-FEM~\citep{knook2024quasi}. These have optimal complexity so-called reverse Cholesky factorizations (Cholesky from the bottom right instead of the top left) for positive definite problems. One of our motivations for the present work is developing an optimal complexity QL factorization for these special block banded matrices, where Givens-based approaches seem to destroy the underlying structure.  The key challenge is generalizing our framework to block BPS matrices while maintaining $O(N)$ complexity. The primary difficulty lies in applying Householder transformations from one block to subsequent blocks in $O(n)$ time (where $n$ is block size and $N$ the total size), rather than $O(n^3)$. While our current framework does not directly apply, the core insight of structure preservation provides a promising foundation for this challenging extension.

\appendix
\section{BPS Structure of the Compact WY Factor $T$}
\label{sec:bps_T}

In this section, we establish that the upper triangular matrix $T \in \mathbb{R}^{n \times n}$ arising from the Compact WY representation $Q = I - V T V^\top$ inherits a BPS structure, with its structural complexity determined entirely by the lower triangular part of the original matrix. Throughout this appendix, we adopt the structural definitions, parameter notations, and structural preservation properties of $F$ established in Section~\ref{sec:main}.

\begin{mytheorem}[BPS Structure of the Compact WY Factor $T$]
Let $A \in \BPS$ be a BPS matrix with lower bandwidth $\ell$ and lower semiseparable rank $r$. The upper triangular factor matrix $T \in \mathbb{R}^{n \times n}$ generated by the Compact WY representation $Q = I - V T V^\top$ is also a BPS matrix:
\[
T \in \mathrm{BPS}_{(0, \, r+\ell), \, (0, \, 0)}^{n \times n}.
\]
Specifically, $T$ has lower/upper bandwidth $\ell_T = m_T = 0$ (its banded component is purely diagonal: $\mathbf{D}_T = \diag(\tau_1, \tau_2, \dots, \tau_{n-1}, 0)$), lower semiseparable rank $r_T = 0$, and upper semiseparable rank bounded by $p_T \le r + \ell$. That is, for any partition index $1 \le k \le n-1$, the off-diagonal subblock satisfies:
\[
\mathrm{rank}\left(T[1:k, \, k+1:n]\right) \le r + \ell.
\]
\end{mytheorem}

\begin{proof}
By Theorem~\ref{thm:structure_preserve}, the Householder factor matrix $F$ stores the Householder vectors $V = [v_1, v_2, \dots, v_{n-1}] \in \mathbb{R}^{n \times (n-1)}$ in its strictly lower triangular part. Because Householder reflections eliminate subdiagonal entries column by column from left to right, $V$ inherits structural properties strictly from the lower triangular portion of $A$, governed by lower bandwidth $\ell$ and lower semiseparable rank $r$ (and remains entirely independent of upper bandwidth $m$ or upper rank $p$). 

Specifically, each Householder vector $v_j \in \mathbb{R}^n$ ($1 \le j \le n-1$) satisfies $v_j[1:j-1] = \mathbf{0}$, $v_j[j] = 1$, and its remaining trailing part is decomposed as:
\[
v_j = \mathbf{e}_j + \bar{U}_j \bar{\mathbf{k}}_j + \mathbf{b}_j,
\]
where $\mathbf{e}_j \in \mathbb{R}^n$ is the $j$-th standard basis vector, $\bar{\mathbf{k}}_j \in \mathbb{R}^r$, and $\bar{U}_j \in \mathbb{R}^{n \times r}$ is a full-sized generator matrix inheriting rows $j+1$ through $n$ from $U \in \mathbb{R}^{n \times r}$ and zero elsewhere:
\[
\bar{U}_j[p, :] = \begin{cases} U[p, :], & \text{if } j+1 \le p \le n, \\ \mathbf{0}, & \text{if } 1 \le p \le j. \end{cases}
\]
The term $\mathbf{b}_j \in \mathbb{R}^n$ is a sparse banded vector supported strictly on indices $j+1$ through $\min(j+\ell, n)$, having at most $\ell$ non-zero entries.

In LAPACK's \texttt{dlarft} bordering algorithm \citep{doi:10.1137/1.9780898719604}, the factor matrix $T$ is built column by column. Let $T_j := T[1\!:\!j, \, 1\!:\!j] \in \mathbb{R}^{j \times j}$ denote the $j \times j$ leading principal submatrix of $T$ for $1 \le j \le n-1$. For step $j$ ($2 \le j \le n-1$), $T_j$ satisfies the bordered update:
\[
T_j = \begin{bmatrix} T_{j-1} & \mathbf{z}_{j-1} \\ \mathbf{0} & \tau_j \end{bmatrix},
\]
where $\mathbf{z}_{j-1} := T[1:j-1, \, j] \in \mathbb{R}^{j-1}$ represents the strictly upper triangular entries of the $j$-th column of $T$, generated by the recurrence relation:
\begin{equation} \label{eq:t_gen_recurrence}
\mathbf{z}_{j-1} = -\tau_j T_{j-1} \mathbf{w}_{j-1}, \quad \text{with } \mathbf{w}_{j-1} := V_{j-1}^\top v_j \in \mathbb{R}^{j-1},
\end{equation}
and $V_{j-1} = \begin{bmatrix} v_1 & v_2 & \dots & v_{j-1} \end{bmatrix} \in \mathbb{R}^{n \times (j-1)}$.

To characterize the inner product vector $\mathbf{w}_{j-1} \in \mathbb{R}^{j-1}$, we evaluate its $i$-th component ($1 \le i \le j-1$) by expanding $v_i$ and $v_j$:
\[
\mathbf{w}_{j-1}[i] = v_i^\top v_j = (\mathbf{e}_i + \bar{U}_i \bar{\mathbf{k}}_i + \mathbf{b}_i)^\top (\mathbf{e}_j + \bar{U}_j \bar{\mathbf{k}}_j + \mathbf{b}_j).
\]

Now, fix a partition index $k$ ($1 \le k \le n-1$) and consider the off-diagonal subblock $T[1:k, \, k+1:n]$. For any column index $m \ge k+1$, we partition $T_{m-1}$ and $\mathbf{w}_{m-1}$ at row/column $k$ into $2 \times 2$ block forms:
\[
T_{m-1} = \begin{bmatrix} T_k & T[1:k, \, k+1:m-1] \\ \mathbf{0}_{(m-1-k) \times k} & T[k+1:m-1, \, k+1:m-1] \end{bmatrix}, \quad
\mathbf{w}_{m-1} = \begin{bmatrix} \mathbf{w}_{m-1}[1:k] \\ \mathbf{w}_{m-1}[k+1:m-1] \end{bmatrix}.
\]
Substituting these partitioned matrices into equation \eqref{eq:t_gen_recurrence} and extracting the top $k$ rows yields:
\[
T[1:k, \, m] = -\tau_m \Big( T_k \, \mathbf{w}_{m-1}[1:k] + T[1:k, \, k+1:m-1] \, \mathbf{w}_{m-1}[k+1:m-1] \Big).
\]
Expanding the product $T[1:k, \, k+1:m-1] \, \mathbf{w}_{m-1}[k+1:m-1]$ as a linear combination of the columns of $T[1:k, \, k+1:m-1]$, we obtain the explicit column expression:
\begin{equation} \label{eq:subblock_col}
T[1:k, \, m] = -\tau_m T_k \, \mathbf{w}_{m-1}[1:k] - \tau_m \sum_{j=k+1}^{m-1} T[1:k, \, j] \, \mathbf{w}_{m-1}[j].
\end{equation}

For any column $m \ge k+1$, truncating the inner product vector $\mathbf{w}_{m-1}$ to its first $k$ components gives $\mathbf{w}_{m-1}[1:k] = V_k^\top v_m$. Expanding both $v_i = \mathbf{e}_i + \bar{U}_i \bar{\mathbf{k}}_i + \mathbf{b}_i$ and $v_m = \mathbf{e}_m + \bar{U}_m \bar{\mathbf{k}}_m + \mathbf{b}_m$, the $i$-th entry of $\mathbf{w}_{m-1}[1:k]$ for $1 \le i \le k < m$ evaluates to:
\[
\mathbf{w}_{m-1}[i] = v_i^\top v_m = (\mathbf{e}_i + \bar{U}_i \bar{\mathbf{k}}_i + \mathbf{b}_i)^\top (\mathbf{e}_m + \bar{U}_m \bar{\mathbf{k}}_m + \mathbf{b}_m).
\]
Since $i < m$, the terms $\mathbf{e}_i^\top \mathbf{e}_m = 0$, $\mathbf{e}_i^\top \bar{U}_m = \mathbf{0}^\top$, and $\mathbf{e}_i^\top \mathbf{b}_m = 0$ disappear because $\bar{U}_m$ and $\mathbf{b}_m$ vanish at row $i$. Expanding the remaining non-zero terms yields:
\[
\mathbf{w}_{m-1}[i] = \bar{\mathbf{k}}_i^\top \bar{U}_i^\top \mathbf{e}_m + \mathbf{b}_i^\top \mathbf{e}_m + \bar{\mathbf{k}}_i^\top \bar{U}_i^\top \bar{U}_m \bar{\mathbf{k}}_m + \bar{\mathbf{k}}_i^\top \bar{U}_i^\top \mathbf{b}_m + \mathbf{b}_i^\top \bar{U}_m \bar{\mathbf{k}}_m + \mathbf{b}_i^\top \mathbf{b}_m.
\]
Because $m \ge k+1 > i$, the generator matrix $\bar{U}_i$ inherits row $m$ as $U[m, :]$, while $\bar{U}_i^\top \bar{U}_m = \sum_{p=m+1}^n U[p, :]^\top U[p, :]$. Defining the tail Gram matrix $S_m := \sum_{p=m+1}^n U[p, :]^\top U[p, :] \in \mathbb{R}^{r \times r}$, all terms starting with $\bar{\mathbf{k}}_i$ simplify to $\bar{\mathbf{k}}_i^\top \boldsymbol{\alpha}_m$ for a vector $\boldsymbol{\alpha}_m \in \mathbb{R}^r$ that is independent of $i$. Stacking these $k$ components across $1 \le i \le k$ yields a low-rank contribution lying in $\text{span}(K_k)$, where $K_k := [\bar{\mathbf{k}}_1, \bar{\mathbf{k}}_2, \dots, \bar{\mathbf{k}}_k]^\top \in \mathbb{R}^{k \times r}$. 

Furthermore, the banded terms $\mathbf{b}_i^\top \mathbf{e}_m$, $\mathbf{b}_i^\top \bar{U}_m \bar{\mathbf{k}}_m$, and $\mathbf{b}_i^\top \mathbf{b}_m$ are non-zero only when $m - i \le \ell$, which contributes exclusively to the bottom $\min(\ell, k)$ entries of $\mathbf{w}_{m-1}[1:k]$.

We explicitly define the fixed $k$-row banded basis matrix $B_{\text{band}}^{(k)} \in \mathbb{R}^{k \times \ell}$ as the trailing $\ell$ columns of the $k \times k$ identity matrix $\mathbf{I}_k$ (padded with $\ell - k$ zero columns if $k < \ell$):
\[
B_{\text{band}}^{(k)} := \begin{bmatrix} \mathbf{0}_{k \times \max(0, \, \ell-k)} & \mathbf{e}_{\max(1, k-\ell+1)}^{(k)} & \dots & \mathbf{e}_k^{(k)} \end{bmatrix} \in \mathbb{R}^{k \times \ell},
\]
where $\mathbf{e}_p^{(k)}$ denotes the $p$-th canonical basis vector in $\mathbb{R}^k$.

Define the subspace $\mathcal{W}_k := \text{span}\left( T_k [K_k \mid B_{\text{band}}^{(k)}] \right) \subset \mathbb{R}^k$. Because $T_k \in \mathbb{R}^{k \times k}$ and $[K_k \mid B_{\text{band}}^{(k)}] \in \mathbb{R}^{k \times (r+\ell)}$ has $r+\ell$ columns, the dimension of $\mathcal{W}_k$ satisfies $\dim(\mathcal{W}_k) \le r + \ell$.

We establish by mathematical induction on $m \in \{k+1, \dots, n\}$ that every column $T[1:k, \, m] \in \mathcal{W}_k$:
\begin{enumerate}[(i)]
    \item Base case ($m = k+1$): The summation in \eqref{eq:subblock_col} is empty, giving $T[1:k, \, k+1] = -\tau_{k+1} T_k \mathbf{w}_k[1:k]$. Since $\mathbf{w}_k[1:k] = V_k^\top v_{k+1} \in \text{span}([K_k \mid B_{\text{band}}^{(k)}])$, there exists a coefficient vector $\boldsymbol{\gamma}_{k+1} \in \mathbb{R}^{r+\ell}$ such that $\mathbf{w}_k[1:k] = [K_k \mid B_{\text{band}}^{(k)}] \boldsymbol{\gamma}_{k+1}$. Left-multiplying by $-\tau_{k+1} T_k$ yields:
    \[
    -\tau_{k+1} T_k \mathbf{w}_k[1:k] = \left( T_k [K_k \mid B_{\text{band}}^{(k)}] \right) \left(-\tau_{k+1} \boldsymbol{\gamma}_{k+1}\right) \in \mathcal{W}_k.
    \]
    \item Inductive step: Assume $T[1:k, \, j] \in \mathcal{W}_k$ for all $k+1 \le j \le m-1$. In equation \eqref{eq:subblock_col}, the leading term $-\tau_m T_k \mathbf{w}_{m-1}[1:k]$ belongs to $\mathcal{W}_k$ because $\mathbf{w}_{m-1}[1:k] = V_k^\top v_m \in \text{span}([K_k \mid B_{\text{band}}^{(k)}])$, which implies $-\tau_m T_k \mathbf{w}_{m-1}[1:k] \in \text{span}\left(T_k [K_k \mid B_{\text{band}}^{(k)}]\right) = \mathcal{W}_k$. Each column $T[1:k, \, j]$ inside the summation belongs to $\mathcal{W}_k$ by the induction hypothesis, and $\mathbf{w}_{m-1}[j]$ acts as a scalar coefficient. Since $\mathcal{W}_k$ is a linear subspace, the linear combination $T[1:k, \, m]$ also lies in $\mathcal{W}_k$.
\end{enumerate}

Because all $n-k$ columns of the subblock $T[1:k, \, k+1:n]$ belong to the $(r+\ell)$-dimensional subspace $\mathcal{W}_k$, the entire subblock factorizes into an outer product of a $k \times (r+\ell)$ matrix and an $(r+\ell) \times (n-k)$ matrix. Therefore, $\mathrm{rank}\left(T[1:k, \, k+1:n]\right) \le r + \ell$, which proves that $T \in \mathrm{BPS}_{(0, \, r+\ell), \, (0, \, 0)}^{n \times n}$.
\end{proof}

\section{Fast RQ Computation for Symmetric BPS Matrices}
\label{sec:fastrq}
The fast QR factorization developed in the previous section provides an efficient direct solver for BPS linear systems. Beyond system inversion, studying the reverse $RQ$ product yields fundamental insights into the algebraic closure of the BPS matrix class under orthogonal similarity transformations, as $RQ = Q^\top A Q$ when $A$ is symmetric. Analyzing this product also establishes a key computational building block for more complex chain-structured orthogonal matrix decompositions. For symmetric BPS matrices, we show that the $RQ$ product strictly preserves the BPS structure, leading to the design of a linear-complexity algorithm for its fast computation.

\subsection{Structure of RQ}

Consider a symmetric BPS matrix $A$ of the form
\begin{equation}\label{express_A_symm}
    A = B_s + \tril(UV^\top, -1) + \triu(VU^\top, 1),
\end{equation}
where $B_s$ is a symmetric banded matrix with lower and upper bandwidth $\ell$, and $U, V \in \mathbb{R}^{n \times r}$ generate the lower and upper semiseparable parts of rank $r$. Let $A=QR$ be its QR factorization. Since $RQ = (Q^{-1}A)Q=Q^\top A Q$ and $A$ is symmetric, $RQ$ is also symmetric.

To describe the structure of intermediate matrices in the computation of $RQ$, we introduce definitions analogous to those in Section \ref{sec:main} but tailored for the symmetric case and the $R$ factor.

\begin{definition}
Denote the set of {\it symmetric banded matrices}  with both lower and upper-bandwidth $\ell$ as $\SBM \subset \bbR^{n \times n}$. In particular, $B \in \SBM \subset \bbR^{n \times n}$ if
 \( B = (b_{kj})_{k,j=1}^n \in \mathbb{R}^{n \times n} \) is a symmetric banded matrix satisfying \( b_{kj}=0 \) for \( k-j>\ell \) or \( j-k>\ell \).
\end{definition}

\begin{definition}
Denote the set of {\it symmetric banded-plus-semiseparable matrices} (SBPS) with lower and upper-semisepable rank $r$ and lower and upper-bandwidth $\ell$  as $\SBPS \subset \bbR^{n \times n}$. In particular, $A \in \SBPS \subset \bbR^{n \times n}$ if
\[
A = B_s + \tril(UV^\top, -1) + \triu(VU^\top, 1)
\]
for \( U, V \in \mathbb{R}^{n\times r} \) and \( B   \in \SBM \).
\end{definition}

Following the $QR$ factorization of $A \in \SBPS$, we let $R$ be the upper triangular factor and $ I - \tau_k \bfy_k \bfy_k^\top$ be the $k$-th Householder reflector. We define the sequence of matrices $\{R_k\}$ such that $R_0 = R$ and $R_k = R_{k-1} (I - \tau_k \bfy_k \bfy_k^\top)$ for $k=1, \dots, n-1$, leading to $R_{n-1} = RQ$. The structural properties of $R_k$ are governed by its trailing columns, which we formally describe using the following definition:

\begin{definition}
Given an upper triangular matrix $R\in \mathbb R^{n\times n}$ and $U\in \mathbb R^{n\times r}$, define the linear space:
\begin{align*}
\calH(R;U) &:= \bigg\{
RU\Omega U^\top + \Phi U^\top + RU\Psi + \Lambda :\\
&\qquad\quad  \Omega \in \mathbb{R}^{r\times r},  \Phi \in \calS_{\ell \times r}^{n \times r},  \Psi \in \calS_{r \times \ell}^{r \times n},  \Lambda \in \calS_{\ell \times \ell}^{n \times n}
\bigg\} \subset \bbR^{n \times n}.
\end{align*}
\end{definition}

For convenience we also define the nonzero portions of these matrices as $\Phi_s:=\Phi[1:\min(\ell,n),:]$, $\Psi_s:=\Psi[:,1:\min(\ell,n)]$, and $\Lambda_s:=\Lambda[1:\min(\ell,n),1:\min(\ell,n)]$.

We characterize the structural evolution of the intermediate matrices $R_k$ as follows:

\begin{definition}
We say $R_k \in {\rm SBPSR}_k(A)$ if for
\[
\tilde U :=  RU \in \bbR^{n \times r},
\]
 there exist 
 $B_k \in {\mathrm{SBM}}_{\ell}^{n \times n}$ and
$V_k \in \bbR^{n \times r}$ 
so that the first $k$ columns below the diagonal satisfy the conditions of the SBPS format:
\begin{align*}
\tril (R_k[:,1\!:\!k]) &= \tril \big((B_k + \tril(\tilde U V_k^\top, -1) + \triu(V_k \tilde U^\top,1))[:,1\!:\!k]\big) 
\end{align*}
whilst the bottom right block is a perturbed upper triangular matrix:
\begin{align*}
R_k[k\!+\!1\!:\!n,k\!+\!1\!:\!n] &= R[k+1\!:\!n,k+1\!:\!n] + H_k,
\end{align*}
where $H_k \in \calH(A[k+1\!:\!n,k+1\!:\!n];U[k+1:n,:])$.
\end{definition}

We will show that the $R_k \in {\rm SBPSR}_k(A)$ by induction.
As a preliminary step we note that the perturbation structure of the bottom right is preserved under truncation:

\begin{lemma}\label{Lemma:PpertubR}
    If $H \in \calH(R;U)$ then $H[k\!:\!n,k\!:\!n] \in \calH(R[k:n,k:n];U[k:n,:])$.
\end{lemma}
\begin{proof}
    We show for $k=2$ with the other cases following by induction. We write:
\begin{align*}
    H[2\!:\!n,2\!:\!n] = &(RU)[2:n,:]\Omega U[2:n,:]^\top   +
    \Phi[2:n,:]U[2:n,:]^\top \\&+ (RU)[2:n,:]\Psi[:,2:n] + \Lambda[2:n,2:n].
\end{align*}

Since $R$ is upper triangular, it follows that
\begin{align*}
    (RU)[2:n,:]=R[2:n,2:n]U[2:n,:],
\end{align*}

and hence
\begin{align*}
H[2\!:\!n,2\!:\!n] = &R[2:n,2:n]U[2:n,:]\Omega U[2:n,:]^\top   +
    \Phi[2:n,:]U[2:n,:]^\top \\&+ R[2:n,2:n]U[2:n,:]\Psi[:,2:n] + \Lambda[2:n,2:n] \\ \in &\calH(R[2:n,k:n];U[2:n,:]).
\end{align*}
\end{proof}

The structural form of the first $k$ columns of $R_k$ below the diagonal is characterized by the following lemma:

\begin{lemma} \label{lemma: RQ_lower}
    For $1\leq k \leq n-1$, suppose $R_{k-1}  \in {\rm SBPSR}_{k-1}(A)$. Then there exists $B_k$ and $V_k$ such that
\[
\tril(R_k[:,1:k]) = \tril\big( (B_k + \tril(\tilde U V_k^\top, -1) + \triu(V_k \tilde U^\top,1))[:,1:k]\big)
\]
\end{lemma}

\begin{proof}
    We begin by defining the symmetric matrix $B_k$ through its lower triangular part:
    \begin{align*}
    \tril (B_k[:,1\!:\!k-1]) &:= \tril (B_{k-1}[:,1\!:\!k-1]). 
    \end{align*}

    Similarly, for $V_k$ we set
    \begin{align*}
        V_k[1:k-1,:] = V_{k-1}[1:k-1,:].
    \end{align*}

    It remains to specify $B_k[k:n,k]$ and $V_k[k,:]$. Note that the entries $B_k[k, k:n]$ are implicitly determined by the symmetry of $B_k$, and any other entries not involved in the current or previous $k$ steps of the induction may be assigned arbitrarily at this stage.

    Since $R_{k-1}  \in {\rm SBPSR}_{k-1}(A)$, we may write
    \begin{align*}
        R_{k-1}[k:n,k:n] = R[k:n,k:n] + H_{k-1}
    \end{align*}
    where
    \begin{align*}
        H_{k-1} := &R[k:n,k:n]U[k:n,:]\Omega_k U[k:n,:]^\top + \Phi_k U[k:n,:]^\top \\&+ R[k:n,k:n]U[k:n,:]\Psi_k + \Lambda_k
    \end{align*}
    with $\Omega_k \in \mathbb R^{r\times r}$, $\Phi_k \in \calS^{(n-k+1)\times r}_{\ell \times r}$, $\Psi_k\in \calS^{r\times(n-k+1)}_{r\times \ell}$, and $\Lambda_k\in \calS^{(n-k+1)\times(n-k+1)}_{\ell \times \ell}$.
    
    We now consider applying the $k$-th Householder reflection $I-\tau_k\bfy_k \bfy_k^\top$ to 
    \begin{align*}
        R_{k-1}[k:n,k:n] =& R[k:n,k:n] + R[k:n,k:n]U[k:n,:](\Omega_k U[k:n,:]^\top + \Psi_k) \\& + \Phi_k U[k:n,:]^\top + \Lambda_k.
    \end{align*}

    Since $R$ is upper triangular, we have
    \begin{align*}
        R[k:n,k:n]U[k:n,:] = (RU)[k:n,:],
    \end{align*}

    which allows us to rewrite
    \begin{align*}
        R_{k-1}[k:n,k:n] =& R[k:n,k:n] + (RU)[k:n,:]\underbrace{(\Omega_k U[k:n,:]^\top + \Psi_k)}_{:=L_1\in \mathbb R^{r\times(n-k+1)}} \\& + \underbrace{\Phi_kU[k:n,:]^\top + \Lambda_k}_{:=L_2\in \calS^{(n-k+1)\times r}_{\ell \times r}}.
    \end{align*}

    Focusing on the first column, we obtain
    \begin{align*}
        R_{k-1}[k:n,k:n] \bfe_1 = \bft_1 + (RU)[k:n,:]\underbrace{(\Omega_kU[k,:]+\Psi_k\bfe_1)}_{:=\bfr_1\in \mathbb R^r}
    \end{align*}
    where
    \begin{align*}
        \bft_1:=R[k:n,k]+\Phi_kU[k,:]+\Lambda_k[:,1] \in \calS_\ell^{n-k+1}.
    \end{align*}

    The Householder vector $\bfy_k$ can be expressed as $\bfy_k=\hat \bfb_k+U[k:n,:]\bar \bfk_k$ for some $\bar \bfk_k\in \mathbb R^r$ from Eq.(\ref{eq:y_k}), and we compute
    \begin{align*}
        R[k:n,k:n]\bfy_k&=\bft_2+R[k:n,k:n]U[k:n,:]\bar \bfk_k\\&
        =\bft_2+(RU)[k:n,:]\bar \bfk_k
    \end{align*}
    with
    \begin{align*}
        \bft_2:=R[k:n,k:n]\hat \bfb_k \in \calS_{\ell+1}^{n-k+1}.
    \end{align*}
    Hence
    \begin{align*}
        R_{k-1}[k:n,k:n]\bfy_k= \underbrace{\bft_2 + L_2\bfy_k}_{=:\bft_3\in\calS_{\ell+1}^{n-k+1}} + (RU)[k:n,:]\underbrace{(\bar \bfk_k+L_1\bfy_k)}_{=:\bfr_2\in \mathbb R^r}.
    \end{align*}

    Therefore, letting $\delta_k:=-\tau_k\bfy_k^\top\bfe_1$, applying the Householder reflection yields
    \begin{align*}
        R_{k-1}[k:n,k:n](I-\tau_k\bfy_k\bfy_k^\top)\bfe_1 = (\bft_1 + \delta_k \bft_2) + RU[k:n,:](\bfr_1+\delta_k\bfr_2)
    \end{align*}

    which leads to the definitions
    \begin{align}
        V_k[k,:] &:= \bfr_1+\delta_k\bfr_2 \in \mathbb R^r, \label{eq:rq_col_1} \\
        B_k[k:n,k]&:=\bft_1 + \delta_k \bft_2 \in \calS_{\ell+1}^{n-k+1}. \label{eq:rq_col_2}
    \end{align}
\end{proof}

The following lemma completes the inductive step by demonstrating that the trailing submatrix of $R_k$ preserves the perturbed structure defined in the space $\calH$:

\begin{lemma} \label{lemma: RQ_sub_matrix}
    For $1\leq k \leq n-1$, suppose $R_{k-1}  \in {\rm SBPSR}_{k-1}(A)$. Then $R_k \in {\rm SBPSR}_k(A)$. 
\end{lemma}

\begin{proof}
    We have already established the desired structure for all entries except the bottom-right block. It therefore remains to show that
    \begin{align*}
        R_k[k\!+\!1\!:\!n,k\!+\!1\!:\!n]& :=\big(R_{k-1}[k\!:\!n,k\!:\!n](I-\tau_k\bfy_k\bfy_k^\top)\big)[2\!:\!n\!-\!k\!+\!1,2\!:\!n\!-\!k\!+\!1]\\
        &\in R[k\!+\!1\!:\!n,k\!+\!1\!:\!n] + \calH(R[k\!+\!1\!:\!n,k\!+\!1\!:\!n];U[k\!+\!1\!:\!n,:]).
    \end{align*}
    By Lemma \ref{Lemma:PpertubR}, since $R_{k-1}[k:n,k:n]\in R[k:n,k:n]+\calH(R[k:n,k:n];U[k:n,:])$, it follows that
    \begin{align*}
        R_{k-1}[k\!+\!1\!:\!n,k\!+\!1\!:\!n] &\in R[k\!+\!1\!:\!n,k\!+\!1\!:\!n]+\calH(R[k\!+\!1\!:\!n,k\!+\!1\!:\!n];U[k\!+\!1\!:\!n,:]).
    \end{align*}

    More explicitly, we may write
    \begin{align*}
        R_{k-1}[k\!+\!1\!:\!n,k\!+\!1\!:\!n]&= R[k\!+\!1\!:\!n,k\!+\!1\!:\!n] \\
        & \qquad +R[k\!+\!1\!:\!n,k\!+\!1\!:\!n]U[k\!+\!1\!:\!n,:]\Omega_k U[k\!+\!1\!:\!n,:]^\top\\
        &\qquad
        +\Phi_k[2\!:\!n\!-\!k\!+\!1,:]U[k\!+\!1\!:\!n,:]^\top\\
        &\qquad +R[k\!+\!1\!:\!n,k\!+\!1\!:\!n]U[k\!+\!1\!:\!n,:]\Psi_k[\!:\!,2\!:\!n\!-\!k\!+\!1]\\
        &\qquad +\Lambda_k[2\!:\!n\!-\!k\!+\!1,2\!:\!n\!-\!k\!+\!1].
    \end{align*}

    It thus remains to analyze the term $R_{k-1}[k:n,k:n]\bfy_k\bfy_k^\top$:

    \begin{align*}
        R_{k-1}[k\!:\!n,k\!:\!n]\bfy_k\bfy_k^\top&=\big(\bft_3+(RU)[k\!:\!n,:]\bfr_2\big)(\hat\bfb_k^\top+\bar\bfk_k^\top U[k\!:\!n,:]^\top)\\
        &=\underbrace{\bft_3\hat\bfb_k^\top}_{\in \calS^{(n-k+1)\times(n-k+1)}_{(\ell+1)\times(\ell+1)}} + \underbrace{\bft_3\bar \bfk_k^\top}_{\in\calS^{(n-k+1)\times r}_{(\ell+1)\times r}} U[k\!:\!n,:]^\top \\
        &\qquad + (RU)[k\!:\!n,:]\underbrace{\bfr_2\hat\bfb_k^\top}_{\in \calS^{r\times(n-k+1)}_{r\times(\ell+1)}} \\
        &\qquad +(RU)[k\!:\!n,:]\underbrace{\bfr_2\bar\bfk_k^\top}_{\in\mathbb R^{r\times r}} U[k\!:\!n,:]^\top.
    \end{align*}

    Using 
    \begin{align*}
        \big((RU)[k:n,:]\big)[2:n-k+1,:] = R[k+1:n,k+1:n]U[k+1:n,:],
    \end{align*}
    we can express
    \begin{align*}
        R_k[k\!+\!1\!:\!n,k\!+\!1\!:\!n]&=R_{k-1}[k\!+\!1\!:\!n,k\!+\!1\!:\!n] \\
        &\qquad -\tau_k\big(R_{k-1}[k\!:\!n,k\!:\!n]\bfy_k\bfy_k^\top)[2\!:\!n\!-\!k\!+\!1,2\!:\!n\!-\!k\!+\!1]\\
        &=
        R[k\!+\!1\!:\!n,k\!+\!1\!:\!n]\\
        &\qquad+R[k\!+\!1\!:\!n,k\!+\!1\!:\!n]U[k\!+\!1\!:\!n,:]\Omega_{k+1}U[k\!+\!1\!:\!n,:]^\top\\
        &\qquad
        +\Phi_{k+1}U[k\!+\!1\!:\!n,:]^\top\\
        &\qquad +R[k\!+\!1\!:\!n,k\!+\!1\!:\!n]U[k\!+\!1\!:\!n,:]\Psi_{k+1}+\Lambda_{k+1}
    \end{align*}
    where
    \begin{align}
        \Omega_{k+1}&:=\Omega_k-\tau_k\bfr_2\bar \bfk_k^\top \in \mathbb R^{r\times r}, \label{eq:rq_submatrix_1}\\
        \Phi_{k+1}&:=\Phi_{k}[2\!:\!n\!-\!k\!+\!1,:]-\tau_k\bft_3[2\!:\!n\!-\!k\!+\!1]\bar \bfk_k^\top \in \calS^{(n-k)\times r}_{\ell \times r}, \label{eq:rq_submatrix_2}\\
        \Psi_{k+1}&:=\Psi_{k}[:,2\!:\!n\!-\!k\!+\!1]-\tau_k\bfr_2\hat \bfb_k^\top [2\!:\!n\!-\!k\!+\!1] \in \calS^{r\times(n-k)}_{r\times \ell}, \label{eq:rq_submatrix_3}, \\
        \Lambda_{k+1}&:=\Lambda_k[2\!:\!n\!-\!k\!+\!1,2\!:\!n\!-\!k\!+\!1]-\tau_k\bft_3[2\!:\!n\!-\!k\!+\!1]\hat \bfb_k^\top[2\!:\!n\!-\!k\!+\!1] \label{eq:rq_submatrix_4}\\ &\in \calS^{(n-k)\times(n-k)}_{\ell \times \ell}.  \notag
    \end{align}
\end{proof}

Combining these inductive results, we arrive at the central theorem of this section: the $RQ$ product preserves the SBPS structure of $A$, including both the semiseparable rank $r$ and the bandwidth $\ell$:

\begin{mytheorem}\label{thm:structure_rq}
    After the QR factorization of a SBPS matrix $A\in {\mathrm{SBPS}}_{r,\ell}^{n \times n}$,
$
RQ \text{ also }\in {\mathrm{SBPS}}_{r,\ell}^{n \times n}.
$

\end{mytheorem}

\begin{proof}
    It is easy to see that $R_0:=R\in {\rm SBPSR}_{0}(A)$ (with the parameter matrices $\Omega, \Phi, \Psi,$ and $\Lambda$ all $\mathbf{0}$). Applying Lemma \ref{lemma: RQ_lower} and \ref{lemma: RQ_sub_matrix}, it follows by induction that

\[
\tril(R_{n-1}[:,1:n-1]) = \tril \big(B_{n-1} + \tril(\tilde U V_{n-1}^\top,-1) + \triu(V_{n-1} \tilde U^\top,1)\big).
\]
Now define $B_n$ by setting $B_n[n,n]=R_{n-1}[n,n]$ and letting it coincide with $B_{n-1}$ elsewhere, and set $V_n:=V_{n-1}$. Then
\[
\tril(R_{n-1}) = \tril \big(B_{n} + \tril(\tilde U V_n^\top,-1) + \triu(V_n \tilde U^\top,1)\big).
\]  
Since $QR=R_{n-1}$ and both sides are symmetric, it follows that
\[
QR = B_n + \tril(\tilde U V_n^\top,-1) + \triu(V_n \tilde U^\top,1).
\]
\end{proof}

\subsection{Fast RQ Multiplication}

Theorem \ref{thm:structure_rq} leads directly to a fast algorithm for computing the $RQ$ product without explicitly forming the dense orthogonal matrix $Q$.

\begin{algorithm}\caption{Fast RQ computation for Symmetric BPS Matrices}\label{algo:fastrq}
\begin{algorithmic}
\State \textbf{Input:} A symmetric BPS matrix $A$ given by its generators: symmetric banded $B_s$ (bandwidth $\ell$), and $U, V \in \mathbb{R}^{n \times r}$ satisfying $A = B_s + \tril(UV^\top, -1) + \triu(VU^\top, 1)$.
\State \textbf{Output:} The $RQ$ product in structured form: symmetric banded matrix $B_n$, and low-rank generators $\tilde U$ and $V_n$.

\Statex
\State 1. \textbf{Compute fast QR factorization:}
\State \quad Run Algorithm \ref{algo:fastqr} on $A$ to obtain the structured factor matrix $F$ and coefficients $\boldsymbol{\tau}$. Extract the structured representation of the $R$ factor and all Householder vectors $\bfy_k$ (with parameters $\bar{\mathbf{k}}_k$, $\mathbf{b}_k$).
\State \quad Compute and store $\tilde U = R U$ using the structured $R$ in $O(n)$.
\Statex
\State 2. \textbf{Initialize parameters:}
\State \quad Set $\Omega \gets \mathbf{0}_{r \times r}$, $\Phi_s \gets \mathbf{0}_{\min(\ell,n) \times r}$, $\Psi_s \gets \mathbf{0}_{r \times \min(\ell,n)}$, $\Lambda_s \gets \mathbf{0}_{\min(\ell,n) \times \min(\ell,n)}$.
\State \quad Initialize $V_n \gets \mathbf{0}_{n \times r}$.
\Statex
\State 3. \textbf{Forward recursion (apply $Q$ from the right):}
\For{$k = 1$ to $n-1$}
    \State a. Retrieve the Householder parameters $\bar{\mathbf{k}}_{k}$ and $\mathbf{b}_{k}$ for step $k$.
    \State b. Update
    $\Omega \gets$ Eq.~(\ref{eq:rq_submatrix_1}),
    $\Phi_s \gets$ Eq.~(\ref{eq:rq_submatrix_2}),
    $\Psi_s \gets$ Eq.~(\ref{eq:rq_submatrix_3}),
    and $\Lambda_s \gets$ Eq.~(\ref{eq:rq_submatrix_4}).
    \State c. Set 
    $V_n[k,:] \gets $ Eq.(\ref{eq:rq_col_1}) and $B_n[k:\min(k+\ell,n),k] \gets $ Eq.(\ref{eq:rq_col_2}).
\EndFor
\Statex
\State 4. \textbf{Final State}
\State \quad Set $B_n[n,n] \gets R_{n-1}[n,n]$
\Statex
\State 5. \textbf{Return} $B_n$, $\tilde U$, $V_n$.
\end{algorithmic}
\end{algorithm}

\textbf{Complexity analysis.} Step 1, the QR factorization requires $O(n)$ operations. The computation of $\tilde U = RU$ can be performed in $O(nr)$ time due to the structure of $R$. The forward recursion in Step 3 performs a constant amount of work per iteration, as all matrix operations either involve matrices whose dimensions (e.g., $r \times r$, $r \times \ell$, $\ell \times \ell$) are independent of $n$, or involve $U[k:n,:]^\top U[k:n,:]$, which can be evaluated in $O(1)$ time after an $O(n)$ time preprocessing step to build a lookup table. Therefore, Algorithm \ref{algo:fastrq} runs in $O(n)$ time and uses $O(n)$ storage.

\begin{remark}\label{rem:symmetric_tracking}
A subtle but crucial point in Algorithm \ref{algo:fastrq} is that during the forward recursion, we track only the evolving structure of the submatrix $R_{j}[j+1:n, j+1:n]$, even though the entire matrix $R_{j}[:, j+1:n]$ is being modified. More precisely, after applying the first $j$ Householder transformations from the right, the first $j$ columns of $R_{j}$ have reached their final state and will not change in subsequent steps.
However, the first $j$ rows of $R_{j}$ (in columns $j+1:n$) are still subject to modification by later transformations.
Nevertheless, because the final product $RQ$ is known to be symmetric, these pending row entries are not independent: they must eventually equal the corresponding entries in the already-fixed columns.
This allows the algorithm to safely disregard the explicit updating of these rows and focus solely on the lower-right square submatrix, which is the only part whose future evolution is not predetermined.
\end{remark}

\begin{remark}[Perturbation and Shift Invariance]
The structural preservation results extend directly to algebraic shifts applied to the matrix. Given a scalar shift $\mu \in \mathbb{R}$, the shifted matrix $A - \mu I = (B_s - \mu I) + \tril(UV^\top, -1) + \triu(VU^\top, 1)$ remains a symmetric BPS matrix, requiring only a standard diagonal correction to the banded part $B$. The resulting factored product can then be adjusted back to obtain $RQ + \mu I$, which naturally inherits the identical BPS formatting boundaries. Hence, the algebraic closure of the BPS class under these linear translations remains intact and can be computed entirely within $O(n)$ time.
\end{remark}

Figure~\ref{fig:rq_scaling} confirms the $O(n)$ complexity of computing the $RQ$ product for symmetric BPS matrices. The four parameter sets exhibit linear scaling, with all curves scaling parallel to the $O(n)$ reference line, thereby validating the algebraic and complexity analysis presented in this Appendix.

\begin{figure}[htbp]
  \centering
  \begin{minipage}{0.48\textwidth}
    \centering
    \includegraphics[width=\textwidth]{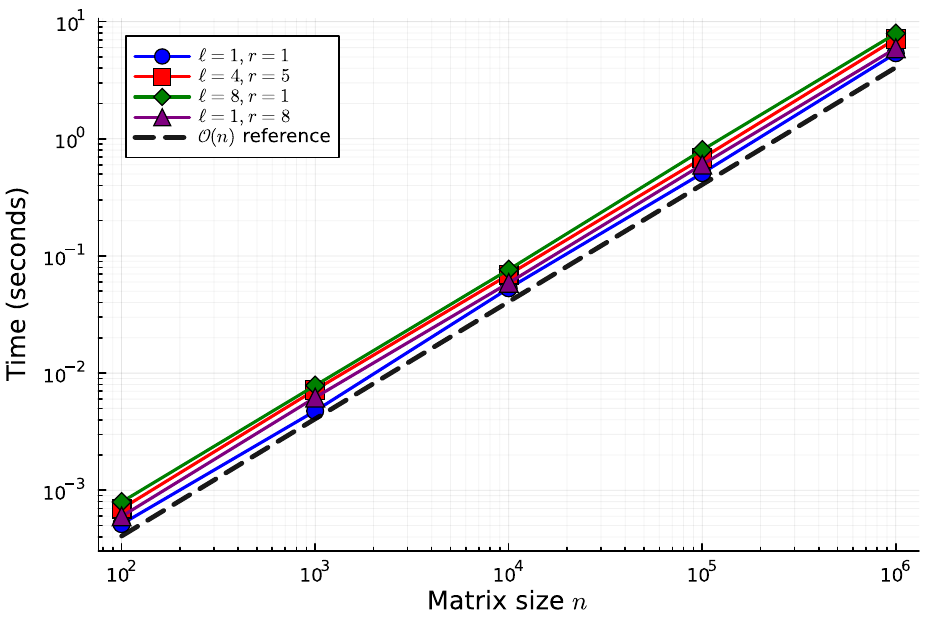}
    \vspace{0.1cm}
    \\ \textbf{(a)} Isolated $RQ$ product loop.
    \label{fig:only_rq}
  \end{minipage}
  \hfill
  \begin{minipage}{0.48\textwidth}
    \centering
    \includegraphics[width=\textwidth]{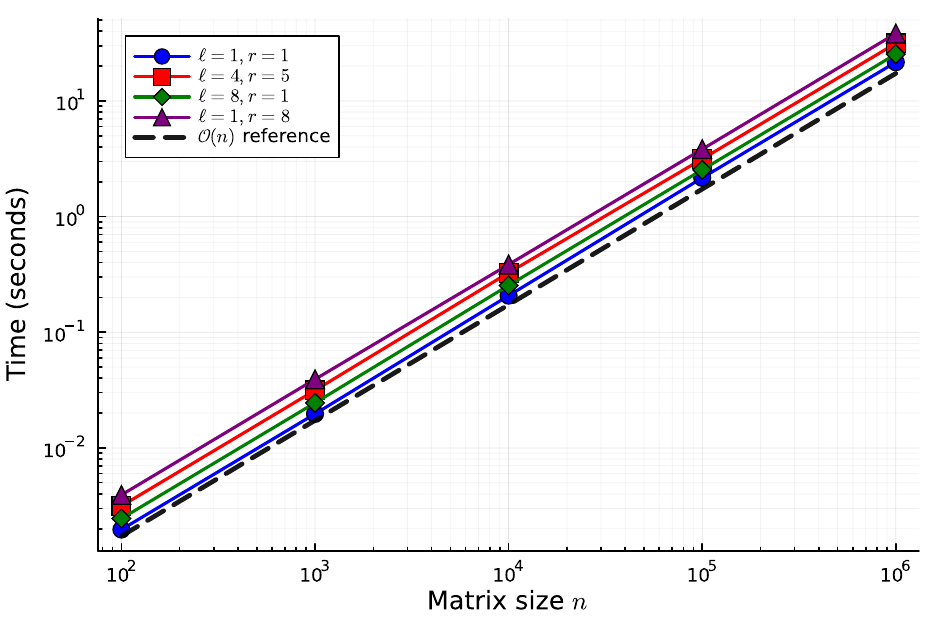}
    \vspace{0.1cm}
    \\ \textbf{(b)} Combined $QR$ and $RQ$.
    \label{fig:qr_plus_rq}
  \end{minipage}
  \vspace{0.3cm}
  \caption{Log-log plots of computation times for symmetric BPS matrices versus $n$ across four distinct parameter sets $(\ell, r)$. The left panel (a) illustrates the isolated $RQ$ computation cost, while the right panel (b) presents the cumulative execution time encompassing both the complete structured $QR$ factorization and the subsequent $RQ$ product. The dashed reference lines possess a slope of 1, indicating strict $O(n)$ linear scaling for both components.}
  \label{fig:rq_scaling}
\end{figure}

\printcredits

\section*{Declaration of Generative AI in the writing process}

During the preparation of this work, the authors used AI-assisted technologies (Gemini) solely for English language editing, grammatical proofreading, and improving text readability. The technology was not used to produce code, execute numerical simulations, perform mathematical derivations, or generate scientific data. After using these tools, the authors reviewed and edited the content as needed and take full responsibility for the content of the publication.

\bibliographystyle{cas-model2-names}

\bibliography{cas-refs}

@article{knook2024quasi,
  title={Quasi-optimal complexity $ hp $-{FEM} for {Poisson} on a rectangle},
  author={Knook, Kars and Olver, Sheehan and Papadopoulos, Ioannis},
  journal={IMA Journal of Numerical Analysis},
  pages={draf102},
  year={2025}
}

@article{olver2013fast,
  title={A fast and well-conditioned spectral method},
  author={Olver, Sheehan and Townsend, Alex},
  journal={SIAM Review},
  volume={55},
  number={3},
  pages={462--489},
  year={2013},
  publisher={SIAM}
}

@misc{OlverNIST,
  title ={{NIST} {D}igital {L}ibrary of {M}athematical {F}unctions},
  howpublished ={http://dlmf.nist.gov/, Release 1.1.4 of 2022-01-15},
  url ={http://dlmf.nist.gov/},
  author={F.~W.~J. Olver and A.~B. {Olde Daalhuis} and D.~W. Lozier and B.~I. Schneider and
                R.~F. Boisvert and C.~W. Clark and B.~R. Miller and B.~V. Saunders and
                H.~S. Cohl and M.~A. McClain}
}

@article{chandrasekaran2003fast,
  title={Fast and stable algorithms for banded plus semiseparable systems of linear equations},
  author={Chandrasekaran, Shiv and Gu, Ming},
  journal={SIAM Journal on Matrix Analysis and Applications},
  volume={25},
  number={2},
  pages={373--384},
  year={2003},
  publisher={SIAM}
}

@article{chandrasekaran2006fast,
  title={A fast {ULV} decomposition solver for hierarchically semiseparable representations},
  author={Chandrasekaran, Shiv and Gu, Ming and Pals, Timothy},
  journal={SIAM Journal on Matrix Analysis and Applications},
  volume={28},
  number={3},
  pages={603--622},
  year={2006},
  publisher={SIAM}
}

@article{xia2010fast,
  title={Fast algorithms for hierarchically semiseparable matrices},
  author={Xia, Jianlin and Chandrasekaran, Shivkumar and Gu, Ming and Li, Xiaoye S},
  journal={Numerical Linear Algebra with Applications},
  volume={17},
  number={6},
  pages={953--976},
  year={2010},
  publisher={Wiley Online Library}
}

@article{vandebril2005implicit,
  title={An implicit {QR} algorithm for symmetric semiseparable matrices},
  author={Vandebril, Raf and Van Barel, Marc and Mastronardi, Nicola},
  journal={Numerical Linear Algebra with Applications},
  volume={12},
  number={7},
  pages={625--658},
  year={2005},
  publisher={Wiley Online Library}
}

@article{van2004two,
  title={Two fast algorithms for solving diagonal-plus-semiseparable linear systems},
  author={Van Camp, Ellen and Mastronardi, Nicola and Van Barel, Marc},
  journal={Journal of Computational and Applied Mathematics},
  volume={164},
  pages={731--747},
  year={2004},
  publisher={Elsevier}
}

@article{delvaux2008qr,
  title={A {QR}-based solver for rank structured matrices},
  author={Delvaux, Steven and Van Barel, Marc},
  journal={SIAM Journal on Matrix Analysis and Applications},
  volume={30},
  number={2},
  pages={464--490},
  year={2008},
  publisher={SIAM}
}

@article{chandrasekaran2007fast,
  title={A fast solver for {HSS} representations via sparse matrices},
  author={Chandrasekaran, Shiv and Dewilde, Patrick and Gu, Ming and Lyons, William and Pals, Timothy},
  journal={SIAM Journal on Matrix Analysis and Applications},
  volume={29},
  number={1},
  pages={67--81},
  year={2007},
  publisher={SIAM}
}

@article{massei2020hm,
  title={hm-toolbox: {MATLAB} software for {HODLR} and {HSS} matrices},
  author={Massei, Stefano and Robol, Leonardo and Kressner, Daniel},
  journal={SIAM Journal on Scientific Computing},
  volume={42},
  number={2},
  pages={C43--C68},
  year={2020},
  publisher={SIAM}
}

@article{vandebril2005note,
  title={A note on the representation and definition of semiseparable matrices},
  author={Vandebril, Raf and Van Barel, Marc and Mastronardi, Nicola},
  journal={Numerical Linear Algebra with Applications},
  volume={12},
  number={8},
  pages={839--858},
  year={2005},
  publisher={Wiley Online Library}
}

@article{delvaux2008givens,
  title={A {Givens-weight} representation for rank structured matrices},
  author={Delvaux, Steven and Van Barel, Marc},
  journal={SIAM Journal on Matrix Analysis and Applications},
  volume={29},
  number={4},
  pages={1147--1170},
  year={2008},
  publisher={SIAM}
}

@article{vandebril2008rational,
  title={Rational {QR}-iteration without inversion},
  author={Vandebril, Raf and Van Barel, Marc and Mastronardi, Nicola},
  journal={Numerische Mathematik},
  volume={110},
  number={4},
  pages={561--575},
  year={2008},
  publisher={Springer}
}

@article{mastronardi2001fast,
  title={Fast and stable two-way algorithm for diagonal plus semi-separable systems of linear equations},
  author={Mastronardi, Nicola and Chandrasekaran, Shivkumar and Van Huffel, Sabine},
  journal={Numerical linear algebra with applications},
  volume={8},
  number={1},
  pages={7--12},
  year={2001},
  publisher={Wiley Online Library}
}

@article{eidelman1997inversion,
  title={Inversion formulas and linear complexity algorithm for diagonal plus semiseparable matrices},
  author={Eidelman, Y and Gohberg, I},
  journal={Computers \& Mathematics with Applications},
  volume={33},
  number={4},
  pages={69--79},
  year={1997},
  publisher={Elsevier}
}

@article{eidelman2005qr,
  title={The {QR} iteration method for {Hermitian} quasiseparable matrices of an arbitrary order},
  author={Eidelman, Yuli and Gohberg, Israel and Olshevsky, Vadim},
  journal={Linear Algebra and its Applications},
  volume={404},
  pages={305--324},
  year={2005},
  publisher={Elsevier}
}

@article{fasino2005rational,
  title={Rational {Krylov} matrices and {QR} steps on {Hermitian} diagonal-plus-semiseparable matrices},
  author={Fasino, Dario},
  journal={Numerical linear algebra with applications},
  volume={12},
  number={8},
  pages={743--754},
  year={2005},
  publisher={Wiley Online Library}
}

@article{delvaux2006rank,
  title={Rank structures preserved by the {QR}-algorithm: the singular case},
  author={Delvaux, Steven and Van Barel, Marc},
  journal={Journal of Computational and Applied Mathematics},
  volume={189},
  number={1-2},
  pages={157--178},
  year={2006},
  publisher={Elsevier}
}

@inproceedings{chandrasekaran2002fast,
  title={Fast stable solver for sequentially semi-separable linear systems of equations},
  author={Chandrasekaran, Shiv and Dewilde, Patrick and Gu, Ming and Pals, T and van der Veen, Alle-Jan},
  booktitle={International Conference on High-Performance Computing},
  pages={545--554},
  year={2002},
  organization={Springer}
}

@article{delvaux2006structures,
  title={Structures preserved by the {QR}-algorithm},
  author={Delvaux, Steven and Van Barel, Marc},
  journal={Journal of Computational and Applied Mathematics},
  volume={187},
  number={1},
  pages={29--40},
  year={2006},
  publisher={Elsevier}
}

@article{iserles2025stable,
  title={Stable spectral methods for time-dependent problems and the preservation of structure},
  author={Iserles, Arieh},
  journal={Foundations of Computational Mathematics},
  volume={25},
  number={2},
  pages={683--723},
  year={2025},
  publisher={Springer}
}

@book{doi:10.1137/1.9780898719604,
author = {Anderson, E. and Bai, Z. and Bischof, C. and Blackford, L. S.  and Demmel, J. and Dongarra, J. and Du Croz, J. and Greenbaum, A. and Hammarling, S. and McKenney, A. and Sorensen, D.},
title = {LAPACK Users' Guide},
publisher = {Society for Industrial and Applied Mathematics},
year = {1999},
edition   = {Third},
}

@article{rozsa1991band,
  title={On band matrices and their inverses},
  author={R{\'o}zsa, P{\'a}l and Bevilacqua, Roberto and Romani, Francesco and Favati, Paola},
  journal={Linear Algebra and Its Applications},
  volume={150},
  pages={287--295},
  year={1991},
  publisher={Elsevier}
}

@software{SemiseparableMatrices2024,
  author       = {Chen, Tao and Olver, Sheehan  and others},
  title        = {{SemiseparableMatrices.jl}},
  year         = {2026},
  url          = {https://github.com/JuliaLinearAlgebra/SemiseparableMatrices.jl}
}

@book{vandebril2008matrix,
  title={Matrix Computations and Semiseparable Matrices: Linear Systems},
  author={Vandebril, Raf and Van Barel, Marc and Mastronardi, Nicola},
  volume={1},
  year={2008},
  publisher={JHU Press}
}

@book{eidelman2014separable,
  title={Separable Type Representations of Matrices and Fast Algorithms},
  author={Eidelman, Yuli and Gohberg, Israel and Haimovici, Iulian},
  volume={1},
  year={2014},
  publisher={Springer}
}

@article{vandebrilmanual,
  title={Manual for {SSPACK}: the semiseparable software package},
  author={Vandebril, Raf and Van Barel, Marc},
  year={2010}
}

@article{colbrook2021computing,
  title={Computing spectral measures of self-adjoint operators},
  author={Colbrook, Matthew and Horning, Andrew and Townsend, Alex},
  journal={SIAM review},
  volume={63},
  number={3},
  pages={489--524},
  year={2021},
  publisher={SIAM}
}

@techreport{schreiber1987storage,
  title={A storage efficient WY representation for products of Householder transformations},
  author={Schreiber, Robert S and Van Loan, Charles},
  year={1987},
  institution={Cornell University}
}

\end{document}